\documentclass[a4paper,12pt,reqno]{amsart}
\usepackage[T1]{fontenc}
\usepackage[utf8]{inputenc}
\usepackage{bera}
\usepackage{amssymb,dsfont,mathrsfs}
\usepackage[margin=1in]{geometry}
\usepackage{enumitem}
\usepackage{graphicx}
\usepackage{pdflscape,afterpage}
\usepackage[bookmarksdepth=2]{hyperref}

\numberwithin{equation}{section}

\newtheorem{theorem}{Theorem}[section]

\newtheorem{corollary}[theorem]{Corollary}
\newtheorem{proposition}[theorem]{Proposition}

\theoremstyle{definition}
\newtheorem{definition}[theorem]{Definition}

\newtheorem{remark}[theorem]{Remark}

\DeclareMathOperator{\re}{Re}

\DeclareMathOperator{\sign}{sign}

\DeclareMathOperator{\ind}{\mathds{1}}

\newcommand{\op}{L}
\newcommand{\dn}{K}

\newcommand{\pr}{\mathds{P}}
\newcommand{\ex}{\mathds{E}}
\newcommand{\exc}{\mathbf{n}}
\newcommand{\exit}{\mathbf{N}}
\newcommand{\C}{\mathds{C}}
\newcommand{\R}{\mathds{R}}

\newcommand{\ph}{\varphi}
\newcommand{\eps}{\varepsilon}
\newcommand{\thet}{\vartheta}

\newcommand{\qv}[1]{\langle #1 \rangle}

\renewcommand{\le}{\leqslant}
\renewcommand{\ge}{\geqslant}

\usepackage{ifthen}
\newcommand{\formula}[2][nolabel]%
{%
 \ifthenelse{\equal{#1}{nolabel}}%
 {\begin{align*} #2 \end{align*}}%
 {%
  \ifthenelse{\equal{#1}{}}%
  {\begin{align} #2 \end{align}}%
  {\begin{align} \label{#1} #2 \end{align}}%
 }%
}

\newcommand{\ignore}[1]{}

\begin{document}

\title[Boundary traces of shift-invariant diffusions in half-plane]{Boundary traces of shift-invariant \\ diffusions in half-plane}
\author{Mateusz Kwaśnicki}
\thanks{Work supported by the Polish National Science Centre (NCN) grant no.\@ 2015/19/B/ST1/01457}
\address{Mateusz Kwaśnicki \\ Faculty of Pure and Applied Mathematics \\ Wrocław University of Science and Technology \\ ul. Wybrzeże Wyspiańskiego 27 \\ 50-370 Wrocław, Poland}
\email{mateusz.kwasnicki@pwr.edu.pl}
\date{\today}

\begin{abstract}
We study boundary traces of shift-invariant diffusions: two-dimensional diffusions in the upper half-plane $\R \times [0, \infty)$ (or in $\R \times [0, R)$) invariant under horizontal translations. We prove that the corresponding trace processes are Lévy processes with completely monotone jumps, and, conversely, every Lévy process with completely monotone jumps is a boundary trace of some shift-invariant diffusion. Up to some natural transformations of space and time, this correspondence is bijective. We also reformulate this result in the language of additive functionals of the Brownian motion in $[0, \infty)$ (or in $[0, R)$), and Brownian excursions. Our main tool is the recent extension of Krein's spectral theory of strings, due to Eckhardt and Kostenko.
\end{abstract}

\maketitle

%
%                            ---------- o ----------
%

\section{Introduction}
\label{sec:intro}

Building upon recent work~\cite{ek} of Eckhardt and Kostenko, the author identified in~\cite{k:hx} Dirichlet-to-Neumann operators corresponding to appropriate elliptic equations in the half-plane with a class of non-local operators with completely monotone kernels. Here we apply this result to study boundary traces of two-dimensional diffusions in a half-plane $\R \times (0, \infty)$ which are invariant under translations parallel to the boundary. Roughly speaking, we prove that such traces are real-valued Lévy processes with completely monotone jumps, and conversely, to every Lévy process with completely monotone jumps there corresponds a (unique, up to some natural transformations) diffusion in the half-plane; see Theorem~\ref{thm:main}. Our result can be rephrased in the language of functionals of Brownian local times or excursions; see Corollaries~\ref{cor:loc} and~\ref{cor:exc}. These results extend a number of previous works in the area, discussed after the statement of our main results.

We begin with the definition of the class of diffusions considered in this paper.

\begin{definition}
\label{def:diff}
Suppose that $(X(t), Y(t))$ is a stochastic process with values in $\R \times [0, R)$ for some $R \in (0, \infty]$. Let $\pr^{(x,y)}$ denote the probability corresponding to the process started at $(X(0), Y(0)) = (x, y)$. We say that $(X(t), Y(t))$ is a \emph{shift-invariant diffusion} if it is a continuous strong Markov process in $\R \times [0, R)$ with the following properties:
\begin{enumerate}[label=(\alph*)]
\item the law of $(p + X(t), Y(t))$ (under $\pr^{(x, y)}$) is the same as the law of $(X(t), Y(t))$ (under $\pr^{(p + x, y)}$) for every $p \in \R$;
\item if the life-time of the process $Y(t)$ is finite, then $Y(t)$ approaches $R$ at its life-time.
\end{enumerate}
We say that $(X(t), Y(t))$ is a \emph{regular shift-invariant diffusion} if in addition:
\begin{enumerate}[resume*]
\item $Y(t)$ is the Brownian motion (reflected at $0$ and, if $R \in (0, \infty)$, killed at $R$);
\item $X(t)$ is a local martingale.
\end{enumerate}
\end{definition}

Note that every shift-invariant diffusion $(X(t), Y(t))$ is a continuous Markov additive process, with the regulator part $Y(t)$ and the additive part $X(t)$.

Recall that a Lévy process is a one-dimensional stochastic process $Z(t)$ started at $0$, with independent and stationary increments, and càdlàg paths. We allow $Z(t)$ to be killed at a uniform rate (or, equivalently, run only up to an independent, exponentially distributed random life-time). We say that $Z(t)$ has \emph{completely monotone jumps} if the Lévy measure of $Z(t)$ (which describes the intensity of jumps of given magnitude in the paths of $Z(t)$) has a density function $\nu(x)$ on $\R \setminus \{0\}$, and $\nu(x)$ and $\nu(-x)$ are completely monotone functions of $x \in (0, \infty)$.

The following is the main result of this paper.

\begin{theorem}
\label{thm:main}
If $(X(t), Y(t))$ is a shift-invariant diffusion in $\R \times [0, R)$ as in Definition~\ref{def:diff}, $L(t)$ is the local time of $Y(t)$ at $0$ and $L^{-1}(t)$ denotes the right-continuous generalised inverse of $L(t)$, then the trace $Z(t) = X(L^{-1}(t))$ left by $(X(t), Y(t))$ on the boundary $\R \times \{0\}$ is a Lévy process with completely monotone jumps.

Conversely, if $Z(t)$ is a Lévy process with completely monotone jumps, then there is a unique regular shift-invariant diffusion $(X(t), Y(t))$ as in Definition~\ref{def:diff} such that $Z(t)$ is equal in law to the trace $X(L^{-1}(t))$ left by $(X(t), Y(t))$ on the boundary $\R \times \{0\}$.
\end{theorem}

Observe that a regular shift-invariant diffusion has infinite life-time if and only if $R = \infty$. More precisely, with a proper normalisation of the local time $L(t)$, when $(X(t), Y(t))$ is a regular shift-invariant diffusion, then the corresponding trace $Z(t)$ is killed at a uniform rate equal to $1 / R$.

Theorem~\ref{thm:main} can be rephrased in terms of functionals of the one-dimensional Brownian motion in $[0, \infty)$, or excursions of this process.

\begin{corollary}
\label{cor:loc}
Suppose that $R \in (0, \infty]$, $a(dy)$ is a locally finite measure on $[0, R)$, and $b(y)$ is a locally square-integrable function on $[0, R)$. Let $X(t)$ and $Y(t)$ be independent Brownian motions in $\R$ and $[0, \infty)$, respectively, and suppose that $Y(t)$ is reflected at $0$. Let $L_y(t)$ denote the continuous local time of $Y(t)$ at $y$, and let $T_R$ be the hitting time of $R$ for $Y(t)$ (if $R = \infty$, we understand that $T_R = \infty$). Finally, let $Z(t)$ be a stochastic process with lifetime $L_0(T_R)$, given for $t < L_0(T_R)$ by
\formula[eq:thm:loc]{
 Z(t) & = X\biggl(\int_{[0, R)} L_y(L_0^{-1}(t)) a(dy)\biggr) + \int_0^{L_0^{-1}(t)} b(Y(s)) d\dot{Y}(s) ;
}
here $L_0^{-1}(t)$ is the right-continuous generalised inverse of $L_0(t)$, and $\dot{Y}(t) = Y(t) - \tfrac{1}{2} L_0(t)$ is the martingale part in the canonical decomposition of the semi-martingale $Y(t)$. Then $Z(t)$ is a Lévy process with completely monotone jumps.

Conversely, if $Z(t)$ is a Lévy process with completely monotone jumps, then there is a unique number $R \in (0, \infty]$, a unique locally finite measure $a$ on $[0, R)$, and a unique locally square-integrable function $b$ on $[0, R)$ such that $Z(t)$ is equal in law to the process described by the right-hand side of~\eqref{eq:thm:loc}.
\end{corollary}

\begin{corollary}
\label{cor:exc}
Suppose that $R \in (0, \infty]$, $a(dy)$ is a locally finite measure on $[0, R)$, and $b(y)$ is a locally square-integrable function on $[0, R)$. Suppose that $\exc(d\eps)$ is the excursion measure for the Brownian motion in $[0, \infty)$. Denote by $\zeta$ the life-time of the Brownian excursion $\eps(t)$, let $L_y(t)$ denote the continuous local time of $\eps(t)$ at $y$, and let $E$ be the collection of all excursions that never reach the level $R$. Finally, let $X(t)$ be the Brownian motion independent from the excursion process $\eps(t)$, and let $\pr_X$ denote the probability corresponding to $X(t)$. Then the measure:
\formula[eq:thm:exc]{
 \nu(dx) & = \exc\biggl( E ; \ex_X \biggl(X\biggl(\int_{(0, R)} L_y(\zeta) a(dy)\biggr) + \int_0^\zeta b(\eps(s)) d\eps(s) \in dx\biggr) \biggr)
}
is a Lévy measure on $\R \setminus \{0\}$, with a density function $\nu(x)$ such that $\nu(x)$ and $\nu(-x)$ are completely monotone functions of $x > 0$.

Conversely, if $\nu(dx)$ is a Lévy measure with a density function as above, then there is a number $R \in (0, \infty]$, a locally finite measure $a$ on $[0, R)$, and a locally square-integrable function $b$ on $[0, R)$ such that $\nu$ has the representation~\eqref{eq:thm:exc}. If we furthermore require that $R = \infty$ and $a(\{0\}) = 0$, then the measure $a(dy)$ is determined uniquely, while the function $b(y)$ is determined up to addition by a constant.
\end{corollary}

\begin{remark}
If $b(y)$ has locally bounded variation in $[0, R)$, then the Itô integral in Corollary~\ref{cor:loc} can be expressed in terms of the local times $L_y(t)$:
\formula{
 Z(u) & = X\biggl(\int_{[0, R)} L_y(L_0^{-1}(t)) a(dy)\biggr) - \frac{1}{2} \int_0^{L_0^{-1}(t)} L_y(t) db(y) ;
}
see Corollary~\ref{cor:it} below. In the general case, the Itô integral in Corollary~\ref{cor:loc} can be expressed as the generalised integral of the local time:
\formula{
 \int_0^t b(Y(s)) d\dot{Y}(s) & = B(Y(t)) - B(Y(0)) + \frac{1}{2} \int_{(0, R)} b(t) d_y L_y(t) ,
}
where $B(y) = \int_0^y b(x) dx$; we refer to~\cite{rogers:walsch,walsch} for a rigorous discussion. Since $Y(L_0^{-1}(t)) = Y(0) = 0$, we have $B(Y(t)) - B(Y(0)) = 0$, and consequently
\formula{
 Z(u) & = X\biggl(\int_{[0, R)} L_y(L_0^{-1}(t)) a(dy)\biggr) + \frac{1}{2} \int_0^{L_0^{-1}(t)} b(y) d_y L_y(t)
}
when $t < L_0(T_R)$.
\end{remark}

\begin{remark}
As a consequence of Corollary~\ref{cor:loc}, every Lévy process with completely monotone jumps can be approximated by Itô integrals with respect to the martingale part of the Brownian motion $Y(t)$ in $[0, R)$, measured at the inverse local time of $Y(t)$ at $0$. More precisely, there is a sequence $b_n(t)$ of locally square integrable functions on $[0, R)$ such that
\formula{
 Z(t) & = \lim_{n \to \infty} \int_0^{L_0^{-1}(t)} b_n(Y(s)) d\dot{Y}(s) ,
}
with the limit understood in the sense of weak convergence of finite-dimensional distributions. Indeed, it suffices to take any sequence $b_n(t)$ such that $b_n(t)$ converges to $b(t)$ weakly in $L^2([0, r))$ for every $r < R$, and $(b_n(t))^2 dt$ converges vaguely to $a(dt) + (b(t))^2 dt$ on $[0, R)$. A rigorous proof of this result, however, is beyond the scope of this article.
\end{remark}

In principle, our proof of Theorem~\ref{thm:main} identifies the generator of $(X(t), Y(t))$ with an appropriate second-order elliptic differential operator $L$, and the generator of the trace process $Z(t)$ with the Dirichlet-to-Neumann operator $K$ corresponding to $L$. A complete description of such operators $K$ is given by the author in~\cite{k:hx}. The key ingredient in that paper is a recent extension of Krein's spectral theory of strings due to Eckhardt and Kostenko, given in~\cite{ek}; we also refer to~\cite{kl,lw} for less general results of the same kind.

Identification of the generator of the trace left by a diffusion on the boundary of a domain with the Dirichlet-to-Neumann operator corresponding to the generator of the diffusion was already given by Molchanov in~\cite{molchanov1,molchanov2}, under rather restrictive regularity assumptions on the diffusion. Boundary traces of reflecting Brownian motions in smooth domains in $\R^d$ are thoroughly studied in~\cite{hsu}. Representation of traces of symmetric diffusions in terms of the local time is given in~\cite{kolsrud}; see also~\cite{bo}. The case of a diffusion in a half-space with horizontal and vertical components independent is treated in detail in~\cite{ah}. Traces of regular diffusions in irregular domains are discussed in~\cite{bcr}. For results on traces of general symmetric Markov processes, see~\cite{cfy}. Another study of traces of Markov process is given in~\cite{ksv}. However, the author failed to find a general result of that kind, which would include a rather wide class of diffusions introduced in Definition~\ref{def:diff}. For this reason, our argument does not strictly follow the strategy described in the previous paragraph, and we need to work directly with a result from~\cite{ek} (or, strictly speaking, with its reformulation given as an intermediate result in~\cite{k:hx}).

A special case of Theorem~\ref{thm:main}, which links the two-dimensional Brownian motion $(X(t), Y(t))$ in $\R \times [0, \infty)$ with the Cauchy process (symmetric stable process of index $1$) $Z(t)$ on $\R$ was already given by Spitzer in~\cite{spitzer}. In fact, in this result one allows $X(t)$ to be the Brownian motion in $\R^d$. An extension to symmetric stable Lévy processes $Z(t)$ of arbitrary index is a classical result due to Molchanov and Ostrovski, see~\cite{mo}. For analytical counterparts of this result, see~\cite{cs,ms}. An extension of this result to symmetric shift-invariant diffusions and symmetric Lévy processes with continuous jumps was carried out in the analytical context in~\cite{km}. A probalistic approach to the same problem is presented in~\cite{ah}, where in fact the process $X(t)$ is allowed to be an arbitrary diffusion in $\R^d$.

Certain functionals of Brownian local times are studied in detail by Biane and Yor in~\cite{by}. In particular, their work provides a representation of stable Lévy processes in terms of principal values of the local time for the Brownian motion on $\R$ --- rather than usual integrals on $[0, \infty)$, as in Corollary~\ref{cor:loc}. Almost sure convergence and convergence in probability of such principal value integrals is studied in~\cite{cherny}; Remark~(iv) under Theorem~3.2 therein essentially cover a special case of Corollary~\ref{cor:loc} for subordinators (non-negative Lévy processes) $Z(t)$.

We conclude the introduction with a short description of the structure of this article. Section~\ref{sec:pre} briefly introduces main tools used in the proofs of our main results. One-dimensional diffusions and Krein's spectral theory of strings are discussed in the short Section~\ref{sec:one}. The proof of Theorems~\ref{thm:main} and both corollaries is given in Section~\ref{sec:two}. Finally Section~\ref{sec:ex} contains a number of examples.

%
%                            ---------- o ----------
%

\section{Preliminaries}
\label{sec:pre}

We occasionally write $a \wedge b$ for $\min\{a, b\}$. In what follows, we usually spell out explicitly arguments of functions. For example, the symbol $W(t)$ can mean both the entire process (if $t$ is a free variable) and a single random variable (if $t$ is fixed). The meaning will always be clear from the context.

\subsection{Lévy processes with completely monotone jumps}

A \emph{Lévy process} $Z(t)$ is a stochastic process with independent and stationary increments and càdlàg paths, started at $Z(0) = 0$. In this work we only consider real-valued Lévy processes, and we allow a Lévy process to be killed at a uniform rate: there is a killing rate $\gamma \ge 0$ such that $Z(t)$ is equal to the cemetery state with probability $1 - e^{-\gamma t}$. A Lévy process is completely determined by its characteristic exponent $\Psi(\xi)$:
\formula{
 \ex e^{i \xi Z(t)} & = e^{-t \Psi(\xi)}
}
for $t \ge 0$ and $\xi \in \R$, and $\Psi(\xi)$ is given by the Lévy--Khintchine formula:
\formula{
 \Psi(\xi) & = \alpha \xi^2 - i \beta \xi + \gamma + \int_{\R \setminus \{0\}} (1 - e^{i \xi x} + i \xi (1 - e^{-|x|}) \sign x) \nu(dx) .
}
Here $\alpha \ge 0$ is the Gaussian coefficient, $\beta \in \R$ is the drift, $\gamma \ge 0$ is the rate of killing, and $\nu(dx)$ is a non-negative measure on $\R \setminus \{0\}$ such that $\int_{\R \setminus \{0\}} (1 \wedge x^2) \nu(dx)$ is finite --- the Lévy measure of $Z(t)$. For a general account on Lévy processes, we refer to~\cite{applebaum,bertoin,sato}.

Recall that a function $f(x)$ is \emph{completely monotone} on $(0, \infty)$ if $f$ is smooth on $(0, \infty)$ and $(-1)^n f^{(n)}(x) \ge 0$ for every $n = 0, 1, 2, \ldots$ and $x > 0$. By Bernstein's theorem, $f(x)$ is completely monotone on $(0, \infty)$ if and only if $f(x)$ is the Laplace transform of a locally finite non-negative measure in $[0, \infty)$.

A Lévy process $Z(t)$ is said to have \emph{completely monotone jumps} if $\nu(dx)$ is absolutely continuous with respect to the Lebesgue measure, and the density function of $\nu(dx)$, denoted by the same symbol $\nu(x)$, has the following property: $\nu(x)$ and $\nu(-x)$ are completely monotone functions of $x > 0$. This class of Lévy processes first appeared in the work of Rogers, who observed in~\cite{rogers} that Wiener--Hopf factorisation for Lévy processes with completely monotone jumps has a particularly nice structure. Fluctuation theory for this class of processes was further developed by the author in~\cite{k:cm}. We remark that the class of Lévy processes with completely monotone jumps appears naturally in the theory of bell-shape in~\cite{k:bell:1,k:bell:2}; see also~\cite{jk,kk,k:hl,kmr,km,mucha} for other related results. For further discussion of Lévy processes with completely monotone jumps, we refer to~\cite{k:cm}, from where we take the following result.

\begin{proposition}[Theorem~3.3 in~\cite{k:cm}]
\label{prop:rogers}
Suppose that $\Psi(\xi)$ is a continuous function on $\R$, satisfying $\Psi(-\xi) = \overline{\Psi(\xi)}$ for all $\xi \in \R$. The following conditions are equivalent:
\begin{enumerate}[label=\rm (\alph*)]
\item
$\Psi$ is the characteristic exponent of a Lévy process with completely monotone jumps;
\item
for all $\xi \in \R$ we have
\formula{
 \Psi(\xi) & = \alpha \xi^2 - i \beta \xi + \gamma + \frac{1}{\pi} \int_{\R \setminus \{0\}} \biggl(\frac{\xi}{\xi + i s} + \frac{i \xi \sign s}{1 + |s|}\biggr) \frac{\mu(ds)}{|s|}
}
for some non-negative measure $\mu$ on $\R \setminus \{0\}$ which satisfies the integrability condition $\int_{\R \setminus \{0\}} \min\{|s|^{-1}, |s|^{-3}\} \mu(ds) < \infty$;
\item
either $\Psi(\xi) = 0$ for all $\xi \in \R$ or for all $\xi \in \R$ we have
\formula[eq:dn:symbol:exp]{
 \Psi(\xi) & = c \exp\biggl(\frac{1}{\pi} \int_{-\infty}^\infty \biggl(\frac{\xi}{\xi + i s} - \frac{1}{1 + |s|}\biggr) \frac{\thet(s)}{|s|} \, ds\biggr)
}
for some $c > 0$ and some Borel function $\thet$ on $\R$ with values in $[0, \pi]$;
\item
$\Psi$ extends to a holomorphic function in the right complex half-plane $\{\xi \in \C : \re \xi > 0\}$ and $\re (\Psi(\xi) / \xi) \ge 0$ whenever $\re \xi > 0$ (that is, $\Psi(\xi) / \xi$ is a \emph{Nevanlinna--Pick function} in the right complex half-plane).
\end{enumerate}
\end{proposition}

In~\cite{k:cm}, holomorphic extensions of characteristic exponents $\Psi(\xi)$ of Lévy processes with completely monotone jumps are called \emph{Rogers functions}. We follow this terminology here.

Recall that a \emph{subordinator} $Z(t)$ is a non-negative (and hence non-decreasing) Lévy process. In this case it is more convenient to use the Laplace exponent $\Phi(\xi)$ rather than the characteristic (Lévy--Khintchine) exponent $\Psi(\xi)$: we have $\ex e^{-\xi Z(t)} = e^{-t \Phi(\xi)}$ whenever $\re \xi \ge 0$. Clearly, $\Psi(\xi) = \Phi(-i \xi)$ for $\xi \in \R$. Laplace exponents $\Phi(\xi)$ of subordinators with completely monotone jumps are known as \emph{complete Bernstein functions} or \emph{operator monotone functions}.

If $X(t)$ is the Brownian motion, $S(t)$ is a subordinator, and $X(t)$ and $S(t)$ are independent processes, then $Z(t) = X(S(t))$ is a Lévy process called \emph{subordinate Brownian motion}. We will use the following simple fact from~\cite{k:hl}; further discussion can be found in~\cite{k:hx}.

\begin{proposition}[see Proposition~2.14 in~\cite{k:hl}]
\label{prop:sbm}
If $S(t)$ is a subordinator with completely monotone jumps, then the corresponding subordinate Brownian motion $Z(t)$ has completely monotone jumps. Conversely, every symmetric Lévy process with completely monotone jumps is a subordinate Brownian motion corresponding to a subordinator $S(t)$ with completely monotone jumps.
\end{proposition}

\subsection{Additive functionals of the Brownian motion}

A pair of stochastic processes $(X(t), Y(t))$ is said to be a \emph{Markov additive process}, with $Y(t)$ playing the role of the regulator and $X(t)$ being the additive part, if $X(t)$ is a real-valued process, $Y(t)$ and $(X(t), Y(t))$ are strong Markov processes, and the law of $(p + X(t), Y(t))$ under $\pr^{(x, y)}$ is the same as the law of $(X(t), Y(t))$ under $\pr^{(p + x, y)}$. We do not discuss the general theory of Markov additive processes here, and we only cite later in this section the relevant result from~\cite{cinlar:1} and~\cite{cinlar:2}.

If $(X(t), Y(t))$ is a Markov additive process and additionally $X(t)$ is adapted to the natural filtration of $Y(t)$, then $X(t)$ is said to be a (homogeneous, strong Markov) \emph{additive functional} of $Y(t)$. If $X(t)$ is non-decreasing (which is a standing assumption in most works), it is a \emph{positive additive functional}, while if the paths of $X(t)$ are continuous, it is a \emph{continuous additive functional}.

Suppose that $Y(t)$ is the Brownian motion. Every positive continuous additive functional of $Y(t)$ can be written as an integral average of local times of $Y(t)$. More precisely, there exists a family of positive continuous additive functionals $L_y(t)$ of $Y(t)$, which depend continuously on $y \in \R$, and which are characterised by the \emph{Tanaka's formula}: the process $|Y(t) - y| - L_y(t)$ is a martingale. The functional $L_y(t)$ is said to be the \emph{local time} of $Y(y)$ at $y$, and, informally, it measures the amount of time spent by $Y(t)$ in an infinitesimal neighbourhood of $y$ up to time $t$. This is made precise by the following occupation time formula.

\begin{proposition}[Theorem~22.5 in~\cite{kallenberg}]
\label{prop:bm:otf}
If $Y(t)$ is the Brownian motion and $a(y)$ is locally integrable on $\R$, then, with probability $1$,
\formula{
 \int_0^t a(Y(s)) ds & = \int_{-\infty}^\infty a(y) L_y(t) dy .
}
\end{proposition}

We have the following classical representation theorem, due to Volkonsky, McKean and Tanaka.

\begin{proposition}[Theorem~22.25 in~\cite{kallenberg}, Corollary~X.2.13 in~\cite{ry}]
\label{prop:bm:paf}
If $Y(t)$ is the Brownian motion, $L_y(t)$ is the local time of $Y(t)$ at $y$, and $X(t)$ is a positive continuous additive functional of $Y(t)$, then there exists a locally finite non-negative measure $a(dy)$ on $\R$ such that, with probability $1$,
\formula{
 X(t) & = \int_\R L_y(t) a(dy) .
}
Furthermore, if $b(y)$ is a Borel function on $\R$ such that $b(y) a(dy)$ is a locally finite measure on $\R$, then, with probability $1$,
\formula{
 \int_0^t b(Y(s)) dX(s) & = \int_{\R} L_y(t) b(y) a(dy) .
}
\end{proposition}

For further information about local times and positive additive functionals, we refer to Chapter~22 of Kallenberg's monograph~\cite{kallenberg}. Basic references for Brownian motion, its local times and related concepts include~\cite{as,im,kallenberg,ks,rw,rw}. Various properties of Brownian local times can be found in the book~\cite{bs}. The paper~\cite{py} contains a detailed analysis of local times of diffusions from a perspective which is closely related to our work. We also mention~\cite{di}, which dicusses the mean value of the local time of Brownian excursions and meanders up to time $t$.

The following result due to Tanaka provides a complete description of signed continuous additive functionals of the Brownian motion.

\begin{proposition}[see~\cite{tanaka}]
\label{prop:bm:af}
If $Y(t)$ is the Brownian motion and $X(t)$ is a continuous additive functional of $Y(t)$, then, with probability $1$,
\formula{
 X(t) & = \tau(Y(t)) - \tau(Y(0)) + \int_0^t b(Y(s)) dY(s)
}
for a continuous function $\tau(y)$ on $\R$ such that $\tau(0) = 0$, and a locally square integrable function $b(y)$ on $\R$.
\end{proposition}

Noteworthy, if $X(t)$ is a positive continuous additive functional of $Y(t)$, then in the above two representations (Propositions~\ref{prop:bm:paf} and~\ref{prop:bm:af}) $b(y)$ is a convex function, and we have $b(y) = -\tau'(y)$ and $a(dy) = \tfrac{1}{2} \tau''(dy)$ (in the sense of distributions). This is a consequence of the following Itô--Tanaka formula, which is due to Meyer and Wang.

\begin{proposition}[Theorem~22.5 in~\cite{kallenberg}]
\label{prop:bm:it}
If $Y(t)$ is the Brownian motion and $\tau(y)$ is a function such that the second distributional derivative of $\tau(y)$ corresponds to a locally finite measure on $\R$, denoted by $\tau''(dy)$, then
\formula{
 \tau(Y(t)) & = \tau(Y(0)) + \int_0^t \tau'(Y(s)) dY(s) + \frac{1}{2} \int_\R L_y(t) \tau''(dy) .
}
\end{proposition}

We conclude this section with a result on continuous Markov additive processes, proved by by Çinlar; see Theorem~(2.23) in~\cite{cinlar:1} and the discussion following~(1.6) in~\cite{cinlar:2}.

\begin{proposition}[see~\cite{cinlar:1}]
\label{prop:map}
If $(X(t), Y(t))$ is a Markov additive process, where $Y(t)$ is the regulator part and $X(t)$ is the additive part, and additionally $X(t)$ has continuous paths, then
\formula{
 \text{the process $X(t)$ is equal in law to the process $W(A(t)) + B(t)$,}
}
where $W(t)$ is the Brownian motion on $\R$ independent from the process $Y(t)$, $A(t)$ is an appropriate positive continuous additive functional of $Y(t)$, and $B(t)$ is an appropriate (signed) continuous additive functional of $Y(t)$.
\end{proposition}

\subsection{Reflected and killed Brownian motion}

The Brownian motion in $[0, \infty)$, reflected at $0$, can be constructed as the absolute value of the Brownian motion in $\R$. That is, if $\tilde{Y}(t)$ is the Brownian motion in $\R$, then $Y(t) = |\tilde{Y}(t)|$ is the Brownian motion in $[0, \infty)$, reflected at $0$. Furthermore, by Tanaka's formula, if $\tilde{L}_y(t)$ is the local time for $\tilde{Y}(t)$ at $y$, then $|\tilde{Y}(t)| - \tilde{L}_0(t)$ is the martingale part in the canonical decomposition of the semi-martingale $Y(t)$. 

Local times for the Brownian motion in $[0, \infty)$, reflected at $0$, can be constructed using the general theory, as discussed in Chapter~22 in~\cite{kallenberg}. Alternatively, with the notation of the previous paragraph, we can define $L_y(t) = \tilde{L}_y(t) + \tilde{L}_{-y}(t)$ for $y \ge 0$. In particular, one should keep in mind that $L_0(t) = 2 \tilde{L}_0(t)$, and therefore $Y(t) - \tfrac{1}{2} L_0(t)$ is the martingale part in the canonical decomposition of $Y(t)$.

Suppose now that $R \in (0, \infty]$, and that $Y(t)$ is the Brownian motion in $[0, \infty)$, reflected at $0$. We will often consider the Brownian motion in $[0, R)$, reflected at $0$ and, if $R \in (0, \infty)$, killed at $R$. By this we mean the process $Y(t)$ described above, started at a point in $[0, R)$, and run only up to the first hitting time $T_R$ of $R$. More formally, we define
\formula{
 T_R & = \inf\{t \ge 0 : Y(t) \ge R\} ,
}
and we define $Y_R(t)$ to be equal to $Y(t)$ for $t < T_R$, and equal to the cemetery state for $t \ge T_R$. In the remaining part of the article, we simply call $Y_R(t)$ the \emph{Brownian motion in $[0, R)$}, without specifying the boundary conditions. We will also use the notation $Y(t)$ rather then $Y_R(t)$, and explicitly state that we only consider $t < T_R$ where necessary.

Perfect counterparts of Propositions~\ref{prop:bm:paf} (which we now combine with the occupation time formula of Proposition~\ref{prop:bm:otf}) and~\ref{prop:bm:af} hold for the Brownian motion in $[0, R)$. The proof of the first one is an obvious modification of the argument presented in Theorem~22.25 in~\cite{kallenberg}, and so we simply state the result. The other one is proved in similar way as in~\cite{tanaka}, and we only sketch the argument.

\begin{corollary}
\label{cor:paf}
If $R \in (0, \infty]$, $Y(t)$ is the Brownian motion in $[0, R)$, $L_y(t)$ is the local time of $Y(t)$ at $y$, and $X(t)$ is a positive continuous additive functional of $Y(t)$, then there exists a locally finite non-negative measure $a(dy)$ on $[0, R)$ such that, with probability $1$,
\formula{
 X(t) & = \int_{[0, R)} L_y(t) a(dy)
}
for $t < T_R$. Furthermore, if $b(y)$ is a Borel function on $[0, R)$ such that $b(y) a(dy)$ is a locally finite measure on $[0, R)$, then, with probability $1$,
\formula{
 \int_0^t b(Y(s)) dX(s) & = \int_{[0, R)} L_y(t) b(y) a(dy) .
}
Finally, if $a(dy)$ has a locally integrable density function $a(y)$, then, with probability $1$,
\formula{
 X(t) & = \int_0^t a(Y(s)) ds .
}
\end{corollary}

\begin{corollary}
\label{cor:af}
If $R \in (0, \infty]$, $Y(t)$ is the Brownian motion in $[0, R)$, and $X(t)$ is a continuous additive functional of $Y(t)$, then, with probability $1$,
\formula{
 X(t) & = \tau(Y(t)) - \tau(Y(0)) + \int_0^t b(Y(s)) d\dot{Y}(s)
}
for a continuous function $\tau(y)$ on $[0, R)$ such that $\tau(0) = 0$, and a locally square integrable function $b(y)$ on $[0, R)$; here $\dot{Y}(t) = Y(t) - \tfrac{1}{2} L_0(t)$ is the martingale part in the canonical decomposition of the semi-martingale $Y(t)$.
\end{corollary}

\begin{proof}[Sketch of the proof]
Suppose first that $R = \infty$. Changing the probability space if necessary, we may assume that $Y(t) = |\tilde{Y}(t)|$ for a Brownian motion $\tilde{Y}(t)$ in $\R$, and so $X(t)$ is an additive functional of $\tilde{Y}(t)$. By Proposition~\ref{prop:bm:af}, we have
\formula{
 X(t) & = \tau(Y(t)) - \tau(Y(0)) + \int_0^t b(\tilde{Y}(s)) d\tilde{Y}(s) .
}
Furthermore, the functions $\tau(y)$ and $b(y)$ are determined uniquely by $X(t)$ (see~\cite{tanaka}). Observe that if we replace $\tilde{Y}(t)$ by $-\tilde{Y}(t)$, the processes $Y(t)$ and $X(t)$ remain unchanged. Thus, $\tau(y) = \tau(|y|)$ and $b(y) = b(|y|) \sign y$. Finally, $d\tilde{Y}(s) = \sign(\tilde{Y}(s)) dY_s - \tfrac{1}{2} dL_0(s)$, and the desired result follows.

If $R \in (0, \infty)$, the argument is essentially the same, except that we need a variant of Proposition~\ref{prop:bm:af} for the Brownian motion in $(-R, R)$ rather than in $\R$. This is proved exactly in the same way as in~\cite{tanaka}, and so we omit the details.
\end{proof}

The Itô--Tanaka formula given in Proposition~\ref{prop:bm:it} also has its variant for the Brownian motion in $[0, R)$. The proof is a relatively simple modification of the argument given in Theorem~22.5 in~\cite{kallenberg}, and therefore we omit it.

\begin{corollary}
\label{cor:it}
If $R \in (0, \infty]$, $Y(t)$ is the Brownian motion in $[0, R)$, $L_y(t)$ is the local time of $Y(t)$ at $y$, and $\tau(y)$ is a function on $[0, R)$ such that the second distributional derivative of $\tau(y)$ corresponds to a locally finite measure on $[0, R)$, denoted by $\tau''(dy)$, then for $t < T_R$ we have, with probability $1$,
\formula{
 \tau(Y(t)) & = \tau(Y(0)) + \int_0^t \tau'(Y(s)) dY(s) + \frac{1}{2} \int_{(0, R)} L_y(t) \tau''(dy) .
}
\end{corollary}

We conclude this section with the expression for the mean value of the local time of the Brownian motion in $[0, R)$ up to its life-time. It is a straightforward consequence of the occupation time formula in Corollary~\ref{cor:paf}, once the density of the occupation measure is identified with the Green's function for the generator. We omit the proof, and refer to Chapter~24 in~\cite{kallenberg} for further discussion. The result can also be obtained from the expression for the density function of $L_y(T_R)$, see formula~3.2.3.2 in~\cite{bs}.

\begin{proposition}
\label{prop:otd}
If $Y(t)$ is the Brownian motion in $[0, \infty)$, $R > 0$ and $y, z \in [0, R)$, then
\formula{
 \ex^y L_z(T_R) & = (R - y) \wedge (R - z) ,
}
where $\ex^y$ is the expectation corresponding to the process $Y(t)$ started at $Y(0) = y$.
\end{proposition}

\subsection{Inverse local time}

If $A(t)$ is a strictly increasing function on $[0, T)$, then we denote by $A^{-1}(u)$ the inverse function of $A(t)$, defined on $[A(0), A(T^-))$. If $A(t)$ is merely non-increasing on $[0, T)$, we often consider the right-continuous generalised inverse function $A^{-1}(u)$, defined for $u \in [0, A(T^-))$ by
\formula{
 A^{-1}(u) & = \sup \{t \in [0, T) : A(t) \le u\} .
}
If $A(t)$ is a positive additive functional of the Brownian motion (or, more generally, a positive strong Markov additive functional of a Markov process), then the generalised right-continuous inverse $A^{-1}(u)$ is a family of Markov times, and if $A(t)$ is the local time at a given point, then $A^{-1}(u)$ is a subordinator; see, for example, Sections~68 and~74 in~\cite{sharpe}.

\subsection{Excursions and exit systems}
\label{sec:exc}

Let $Y_t$ be the Brownian motion in $[0, \infty)$, reflected at $0$. In this section we assume that $Y(0) = 0$. Let $L_y(t)$ be the local time of $Y(t)$ at $y$, continuous with respect to $y$. Finally, let $L_0^{-1}(u)$ denote the right-continuous generalised inverse of $L_0(t)$.

By an \emph{excursion} of $Y(t)$ (away from the origin), we mean a continuous path $\eps(t)$, defined for $t \ge 0$, such for some $\zeta \ge 0$, called the \emph{life-time} of $\eps(t)$, we have $\eps(0) = 0$, $\eps(t) > 0$ for $t \in (0, \zeta)$, and $\eps(t) = 0$ for $t \ge \zeta$. For every $u$ such that $L_0^{-1}(u^-) < L_0^{-1}(u)$ we define the corresponding excursion as
\formula{
 \eps_u(t) & = Y((\alpha_u + t) \wedge \beta_u) , & \alpha_u & = L_0^{-1}(u^-) , & \beta_u & = L_0^{-1}(u) .
}
Clearly, the lifetime of $\eps_u(t)$ is equal to $\zeta_u = \beta_u - \alpha_u$. For completeness, for the remaining values of $u$, we define $\eps_u(t)$ to be the empty excursion $\eps_u(t) = 0$, with zero life-time.

With probability $1$, the path of $Y(t)$ can be decomposed into a countable number of (non-empty) excursions $\eps_u(t)$. The excursion process $\eps_u$ is a Poisson point process with values in the class of all possible excursions, with intensity measure denoted by $\exc(d\eps)$, and called the \emph{excursion measure}. Furthermore, for every non-negative measurable functional $\Phi(u, \eps)$ of $u \in [0, \infty)$ and an excursion $\eps$, we have
\formula{
 \ex \sum_{u \in [0, \infty)} \Phi(u, \eps_u) & = \int_0^\infty \biggl(\int \Phi(u, \eps) \exc(d\eps)\biggr) du ;
}
here we assume that $\Phi(u, \eps) = 0$ if $\eps$ is the empty excursion. We refer to~\cite{im,kallenberg,ks,ry,sharpe,yy} for further discussion and details.

We record that with the above normalisation, if $R \in (0, \infty)$ and $E_R$ denotes the class of excursions $\eps(t)$ that never reach $R$, then $\exc(E_R^c) = 1 / R$; see, for example, Proposition~XII.3.6 and Exercise~XII.2.10.1 in~\cite{ry}.

If $(X(t), Y(t))$ is a shift-invariant diffusion in $\R \times [0, \infty)$ started at $(0, 0)$, we define an excursion of $(X(t), Y(t))$ (away from $\R \times \{0\}$) to be a continuous path $(\delta(t), \eps(t))$ started at $(0, 0)$, taking values in $\R \times (0, \infty)$ for $t \in (0, \zeta)$, and equal to $(x, 0)$ for $t \ge \zeta$, where $x \in \R$ is arbitrary. Suppose that $(X(t), Y(t))$ is a regular shift-invariant diffusion in $\R \times [0, \infty)$. With probability $1$, the path of $(X(t), Y(t))$ can be decomposed into a countable number of (non-empty) excursions $(\delta_u(t), \eps_u(t))$, defined by
\formula{
 \delta_u(t) & = X((\alpha_u + t) \wedge \beta_u) - X(\alpha_u) , & \eps_u(t) & = Y((\alpha_u + t) \wedge \beta_u)
}
whenever $\alpha_u < \beta_u$. Here $L_0(t)$, $\alpha_u = L_0^{-1}(u^-)$ and $\beta_u = L_0^{-1}(u)$ are defined as in the previous paragraphs, and again we define $\delta_u(t) = \eps_u(t) = 0$ to be the empty excursion whenever $\alpha_u = \beta_u$. The excursion process $(\delta_u, \eps_u)$ for $(X(t), Y(t))$ is clearly an extension of the excursion process $\eps_u$ for $Y(t)$. As before, there is a measure $\exit(d\delta, d\eps)$ such that
\formula{
 \ex \sum_{u \in [0, \infty)} \Phi(u, X(\alpha_u), \delta_u, \eps_u) & = \ex \int_0^\infty \biggl(\int \Phi(u, X(\alpha_u), \delta, \eps) \exit(d\delta, d\eps)\biggr) du
}
for every non-negative measurable functional $\Phi(u, x, \delta, \eps)$ of the excursion $(\delta, \eps)$, the local time $u \in [0, \infty)$, and the initial position $x$; we assume here that $\Phi(u, x, \delta, \eps) = 0$ whenever, $(\delta, \eps)$ is an empty excursion.

The measure $\exit(d\delta, d\eps)$ (or, more precisely, the system of similarly defined measures for excursions starting at $(x, 0)$ for an arbitrary $x \in \R$) is called the \emph{exit system} of $(X(t), Y(t))$. We refer to Theorem~74.12 in~\cite{sharpe} for a general result on exit systems, and to~\cite{maisonneuve} and Section~74 in~\cite{sharpe} for a detailed discussion and further references.

If $(X(t), Y(t))$ is a regular shift-invariant diffusion in $\R \times [0, R)$ for some $R \in (0, \infty)$, then, for notational convenience, we extend $Y(t)$ past its life-time $T_R$ to the Brownian motion in $[0, \infty)$. In this case we have the same decomposition into excursions as above, but this time we terminate at the first excursion of $Y(t)$ which reaches the level $R$. More formally, we have
\formula[eq:exit]{
 \ex \sum_{u \in [0, L_0(T_R))} \Phi(u, X(\alpha_u), \delta_u, \eps_u) & = \ex \int_0^{L_0(T_R)} \biggl(\int \ind_{E_R}(\eps) \Phi(u, X(\alpha_u), \delta, \eps) \exit(d\delta, d\eps)\biggr) du ,
}
where $E_R$ is the class of excursions $\eps(t)$ which never reach $R$.

\subsection{Additive functionals of excursions}

We consider the excursions of the Brownian motion $Y(t)$ in $[0, \infty)$. The local times of an excursion $\eps(t)$ are defined in the same way as for the usual Brownian motion. More precisely, for every $r > 0$, under the excursion measure $\exc(d\eps)$, the process $\eps(r + t)$ is the Brownian motion in $[0, \infty)$ absorbed at $0$, with the initial value given by the appropriate (non-probabilistic) entrance law $\exc(\eps(r) \in dy)$. Therefore, the local times of $\eps(r + t)$ are well-defined for every $r > 0$ outside of an event of zero excursion measure. The local times $L_{\eps,y}(t)$ of $\eps(t)$ are obtained by passing to the limit as $r \to 0^+$.

For the excursions $\eps_u(t)$ of $Y(t)$ and $y > 0$, up to an event of zero probability, we have
\formula{
 L_{\eps_u,y}(t) & = L_y((\alpha_u + t) \wedge \beta_u) - L_y(\alpha_u) .
}
By this we mean that with probability $1$, the above equality holds for every $u$, $t$ and $y$. Therefore, by Fubini's theorem, for every locally finite measure $a(dy)$ on $[0, \infty)$,
\formula{
 \int_{[0, \infty)} L_{\eps_u,y}(t) a(dy) & = \int_{[0, \infty)} L_y((\alpha_u + t) \wedge \beta_u) a(dy) - \int_{[0, \infty)} L_y(\alpha_u) a(dy) .
}
The same formula holds for every $u < L_0(T_R)$ if $a(dy)$ is a locally finite measure on $[0, R)$ for some $R \in (0, \infty)$.

In a similar way, the Itô integral over the excursion $\eps$ is defined in terms of the usual Itô integral for the Brownian motion in $[0, \infty)$. More precisely, for any predictable process $A(t)$ (with respect to the natural filtration of the excursion $\eps(t)$) such that
\formula{
 \exc\biggl(\int_0^\zeta (A(t))^2 dt = \infty\biggr) & = 0 ,
}
we have (with the limit in the sense of convergence in probability)
\formula{
 \int_0^\zeta A(t) d\eps(t) & = \lim_{r \to 0^+} \int_r^\zeta A(t) d\eps(t) ,
}
and the integral in the right-hand side is the usual Itô integral with respect to the Brownian motion $\eps(r + t)$ in $[0, \infty)$, absorbed at $0$, with the initial distribution given by the appropriate entrance law.

As before, for the excursions $\eps_u(t)$ of $Y(t)$ we have, outside of an event of probability zero,
\formula{
 \int_0^{\zeta_u} A(\alpha_u + t) d\eps_u(t) & = \int_{\alpha_u}^{\beta_u} A(t) dY(t) = \int_{\alpha_u}^{\beta_u} A(t) d\dot{Y}(t) ,
}
where $\dot{Y}(t) = Y(t) - \tfrac{1}{2} L_0(t)$ is the martingale part in the canonincal decomposition of the semi-martingale $Y(t)$. The same equality holds for every $u < L_0(T_R)$ if we only assume that
\formula{
 \exc\biggl(E_R , \, \int_0^\zeta (A(t))^2 dt = \infty\biggr) & = 0 ,
}
where $E_R$ is the class of excursions $\eps(t)$ which never reach $R$.

%
%                            ---------- o ----------
%

\section{One-dimensional diffusions and Krein's strings}
\label{sec:one}

The relation between one-dimensional diffusions and second-order differential operators is well understood, and we refer to~\cite{im} for a detailed discussion. Generalised diffusions and Krein's spectral theory of strings are thoroughly discussed in~\cite{dm,knight,kw}. Here we only briefly sketch the relevant part of the theory.

Let $R \in (0, \infty]$, and let $Y(t)$ be a diffusion in $[0, R)$ such that every two states of $[0, R)$ communicate with each other. We assume that $Y(t)$ is not killed inside $[0, R)$. More precisely, we assume that the life-time of $Y(t)$, denoted by $T_R$, is either infinite, or we have $Y(T_R^-) = R$. Such a diffusion is completely described by the corresponding \emph{speed measure} $a(dy)$ and \emph{scale function} $b(y)$. By considering the process $b(Y(t))$ rather than $Y(t)$, we may restrict our attention to diffusions $Y(t)$ with scale function $b(y) = y$. In this case we say that $Y(t)$ is a diffusion in \emph{Brownian scale}; and the procedure described above will be called \emph{change of scale}. We refer to Proposition~VII.3.4 in~\cite{ry} for a detailed statement; see also Chapter~3 in~\cite{im}.

If $Y(t)$ is a diffusion in Brownian scale, then $Y(t)$ is a local martingale up to the hitting time of $0$ or $R$, and it can be represented as the appropriately time-changed Brownian motion in $[0, R)$ (reflected at $0$). More precisely,
\formula{
 Y(t) & = \tilde{Y}(\tilde{A}^{-1}(t)) ,
}
where $\tilde{Y}(t)$ is the Brownian motion in $[0, R)$, $\tilde{A}(t)$ is an increasing continuous additive functional of $\tilde{Y}(t)$, and $\tilde{A}^{-1}(t)$ denotes the inverse of $\tilde{A}(t)$. Here we remark that if we allow $\tilde{A}(t)$ to be merely a positive (i.e.\@ non-decreasing) additive functional, and if $\tilde{A}^{-1}(t)$ denotes the right-continuous generalised inverse of $\tilde{A}(t)$, then $Y(t)$ defined as above is a \emph{generalised diffusion}, or a \emph{gap diffusion}. We refer to Section~5.3 in~\cite{im} or Section~V.1 in~\cite{ry} for further details in the diffusion case, and to~\cite{knight} and Section~3 in~\cite{kw} for a discussion of generalised diffusions.

Let $\tilde{L}_y(t)$ denote the local time of $\tilde{Y}(t)$ at $y \in [0, R)$. By Corollary~\ref{cor:paf}, the additive functional $\tilde{A}(t)$ can be represented as
\formula{
 \tilde{A}(t) & = \int_{[0, R)} \tilde{L}_y(t) \tilde{a}(dy)
}
for some locally finite non-negative measure $\tilde{a}(dy)$ on $[0, R)$. Conversely, any such measure $\tilde{a}(dy)$ gives rise to a positive continuous additive functional $\tilde{A}(t)$, and thus to a generalised diffusion $Y(t)$. The functional $\tilde{A}(t)$ is strictly increasing (and consequently $Y(t)$ is a true diffusion) if and only if $\tilde{a}(dy)$ has full support in $[0, R)$. The measure $\tilde{a}(dy)$ (together with the number $R$) is called \emph{Krein's string}.

Let $L_0(t)$ denote the following variant of the local time of $Y(t)$ at $0$:
\formula{
 L_0(t) & = \tilde{L}_0(\tilde{A}^{-1}(t)) .
}
The right-continuous generalised inverse of $L_0(t)$ is a subordinator, and
\formula{
 L_0^{-1}(t) & = \tilde{A}(\tilde{L}_0^{-1}(t)) .
}
Krein proved an analytical variant of the following result: the subordinator $L_0^{-1}(t)$ has completely monotone jumps, and conversely, every subordinator with completely monotone jumps arises in this way for a unique Krein's string $\tilde{a}(dy)$ (or, equivalently, for a unique generalised diffusion $\tilde{Y}(t)$); see~\cite{kk}. We refer to~\cite{knight} and~\cite{kw} for a detailed discussion and the above probabilistic statement.

Krein's correspondence is very inexplicit, at least in one direction: there is no known procedure to construct Krein's string $\tilde{a}(dy)$ which corresponds to a given subordinator $L_0^{-1}(t)$. The other direction is, at least theoretically, simpler: given $\tilde{a}(dy)$, the Laplace exponent $\psi(\xi)$ of $L_0^{-1}(t)$ can be found by solving the Sturm--Liouville-type ordinary differential equation:
\formula[eq:krein]{
 & \begin{cases}
 \phi''(dy) = \xi \phi(y) a(dy) , \\
 \phi(0) = 1 , \\
 \phi(R^-) = 0
 \end{cases}
}
(with the second derivative of $\phi$ understood in the sense of distributions). In this case $\psi(\xi) = -\phi'(0^+) + \tilde{a}(\{0\}) \xi$. Again we refer to~\cite{knight} and~\cite{kw} for further details; see also Chapter~15 in~\cite{ssv}, or Section~3 and Appendix~A in~\cite{km}.

In the above presentation we used the identification of positive continuous additive functionals $\tilde{A}(t)$ of the Brownian motion $\tilde{Y}(t)$ on $[0, R)$ with generalised diffusions $Y(t)$ on $[0, R)$. In the next section, we will see Krein's result from a different perspective: we will consider the Brownian motion $Y(t)$ with a positive continuous additive functional $A(t)$, and identify this pair with a \emph{symmetric} shift-invariant diffusion. Full statement of Theorem~\ref{thm:main} will require an extension of Krein's result to a more general variant of~\eqref{eq:krein}, studied by Eckhardt and Kostenko in~\cite{ek}.

%
%                            ---------- o ----------
%

\section{Shift-invariant diffusions in half-plane}
\label{sec:two}

\subsection{Reduction}
\label{sec:two:red}

We assume that $(X(t), Y(t))$ is a shift-invariant diffusion in $\R \times [0, R)$, as in Definition~\ref{def:diff}. We are interested in the trace left by $(X(t), Y(t))$ on $\R \times \{0\}$, that is, the process
\formula{
 Z(u) & = X(L_0^{-1}(u)) ,
}
where $L_0(t)$ is the local time of $Y(t)$ at $0$ and $L_0^{-1}(u)$ denotes the right-continuous generalised inverse of $L_0(t)$.

As in the previous section, by change of scale, with no loss of generality we may assume that $Y(t)$ is a diffusion in Brownian scale. That is, we assume that the scale function $b(y)$ corresponding to $Y(t)$ is given by $b(y) = y$: in the general case, we simply consider $(X(t), b(Y(t)))$ rather than $(X(t), Y(t))$. Note that these two processes leave the same trace on $\R \times \{0\}$, perhaps up to a linear change of time caused by various possible normalisations of the local time of $Y(t)$ at $0$.

Again as in the previous section, $Y(t)$ is a time-changed Brownian motion in $[0, R)$ (reflected at $0$): we have
\formula{
 Y(t) & = \tilde{Y}(\tilde{A}^{-1}(t))
}
for some Brownian motion $\tilde{Y}(t)$ in $[0, R)$ and an increasing continuous additive functional $\tilde{A}(t)$ of $\tilde{Y}(t)$.

Let $\tilde{L}_y(t)$ denote, as usual, the local time of $\tilde{Y}(t)$ at~$y$. A variant of the local time of $Y(t)$ at $0$ is given by $L_0(t) = \tilde{L}_0(\tilde{A}^{-1}(t))$, and the process $L_0^{-1}(u) = \tilde{A}(\tilde{L}_0^{-1}(u))$ is a subordinator (with completely monotone jumps), Furthermore,
\formula[eq:timechange]{
 Z(u) & = X(L_0^{-1}(u)) = X(\tilde{A}(\tilde{L}_0^{-1}(u)) .
}
Recall that $\tilde{Y}(t) = Y(\tilde{A}(t))$, and define $\tilde{X}(t) = X(\tilde{A}(t))$. Then $\tilde{A}(t)$ is a continuous, increasing family of Markov times for $Y(t)$, and thus the process $(\tilde{X}(t), \tilde{Y}(t))$ is a shift-invariant diffusion; see Section III.21 in~\cite{rw} for further details. By considering this process, rather than $(X(t), Y(t))$, with no loss of generality we may assume that $\tilde{A}(t) = t$ and $Y(t)$ is the Brownian motion in $[0, R)$. Indeed: by~\eqref{eq:timechange}, the traces $Z(u)$ left on $\R \times \{0\}$ by $(X(t), Y(t))$ and $(\tilde{X}(t), \tilde{Y}(t))$ are identical. This procedure will be called \emph{change of time}.

By change of scale and change of time, with no loss of generality we assume that $Y(t)$ is the Brownian motion on $[0, R)$. As remarked in the introduction, $(X(t), Y(t))$ is a Markov additive process. By Proposition~\ref{prop:map}, the process $X(t) - X(0)$ is equal in law to $W(A(t)) + B(t)$, where $W(t)$ is the Brownian motion on $\R$ which is independent from $Y(t)$, $A(t)$ and $B(t)$ are continuous additive functionals of $Y(t)$, and $A(t)$ is positive. Changing the probability space if necessary, with no loss of generality we may assume that in fact $X(t) - X(0) = W(A(t)) + B(t)$.

By Corollary~\ref{cor:af}, with probability $1$,
\formula{
 B(t) & = \tau(Y(t)) - \tau(Y(0)) + \int_0^t b(Y(s)) d\dot{Y}(s)
}
for a continuous function $\tau$ on $[0, R)$ such that $\tau(0) = 0$, and a locally square integrable function $b$ on $[0, R)$. Here $\dot{Y}(s)$ is the martingale part in the canonical decomposition of the semi-martingale $Y(t)$, that is, $\dot{Y}(t) = Y(t) - \tfrac{1}{2} L_0(t)$, where $L_0(t)$ is the local time of $Y(t)$ at $0$. Observe that the process
\formula{
 X(t) - \tau(Y(t)) & = X(0) + \tau(Y(0)) + W(A(t)) + \int_0^t b(Y(s)) d\dot{Y}(s)
}
(with $t < T_R$) is a local martingale. By considering the process $(X(t) - \tau(Y(t)), Y(t))$ rather than $(X(t), Y(t))$, with no loss of generality we may assume that $\tau(y) = 0$ for every $y \in [0, R)$, and thus $X(t)$ is a local martingale. Indeed: the two processes again leave the same trace on $\R \times \{0\}$. This last simplifying procedure will be called \emph{shearing}.

Change of scale, change of time and shearing allow us to reduce any shift-invariant diffusion to a regular shift-invariant diffusion $(X(t), Y(t))$, without changing the trace left by $(X(t), Y(t))$ on $\R \times \{0\}$. For this reason, in the remaining part of the article we restrict our attention to regular shift-invariant diffusions in $\R \times [0, R)$, where $R \in (0, \infty]$.

The above discussion additionally implies that the horizontal coordinate $X(t)$ of a regular shift-invariant diffusion $(X(t), Y(t))$ on $[0, R)$ has the following representation (possibly after a change of the probability space):
\formula[eq:xt]{
 X(t) & = X(0) + W(A(t)) + B(t),
}
where $W(t)$ is the Brownian motion in $\R$ independent of $Y(t)$, $A(t)$ is a positive continuous additive functional of $Y(t)$, and $B(t)$ is a (signed) continuous additive functional of $Y(t)$ which is a local martingale. Furthermore, by Corollary~\ref{cor:paf}, with probability $1$, the functional $A(t)$ is given in terms of the local time $L_y(t)$ of $Y(t)$ at $y$ as
\formula[eq:at]{
 A(t) & = \int_{[0, R)} L_y(t) a(dy)
}
for some locally finite non-negative measure $a(dy)$ on $[0, R)$. On the other hand, as discussed in the previous paragraph, Corollary~\ref{cor:af} implies that the functional $B(t)$ is the Itô integral
\formula[eq:bt]{
 B(t) & = \int_0^t b(Y(s)) d\dot{Y}(s)
}
for a locally square-integrable function $b(y)$ on $[0, R)$. The notation introduced above: $L_y(t)$, $\dot{Y}(t)$, $A(t)$, $B(t)$, $W(t)$, $a(dy)$ and $b(y)$, is kept until the end of this section.

We remark that the three simplifying procedures described above: change of scale, change of time and shearing, are probabilistic analogues of similar reduction steps in Section~1.1 in~\cite{k:hx}: scaling, normalisation and shearing.

\subsection{Symmetric diffusions}

In this section we focus on a special case, when we assume that the regular shift-invariant diffusion $(X(t), Y(t))$ is additionally invariant under reflections with respect to the vertical axis: the law of $(-X(t), Y(t))$ (under $\pr^{(x,y)}$) is the same as the law of $(X(t), Y(t))$ (under $\pr^{(-x,y)}$).

In this case we necessarily have $b(y) = 0$ for almost all $y \in [0, R)$ in formula~\eqref{eq:xt}, and consequently $X(t) = W(A(t))$. It follows that
\formula{
 Z(u) & = X(L_0^{-1}(u)) = W(A(L_0^{-1}(u))) , && \text{with} & A(t) & = \int_{[0, R)} L_y(t) a(dy) .
}
As explained in Section~\ref{sec:one}, Krein proved that $A(L_0^{-1}(u))$ is a subordinator with completely monotone jumps, and every such subordinator can be realised as $A(L_0^{-1}(u))$ for a unique Krein's string $a(dy)$. Therefore, $Z(u)$ is a subordinate Brownian motion, corresponding to a subordinator with completely monotone jumps. By Proposition~\ref{prop:sbm}, this class of processes is precisely the class of symmetric Lévy processes with completely monotone jumps.

The above paragraph proves Theorem~\ref{thm:main} under the additional symmetry assumption. The analytical counterpart of this result is discussed in detail in~\cite{km}. In the next section, we study the general case, and our proof is based on an extension of Krein's result proved by Eckhardt and Kostenko in~\cite{ek}, or, more precisely, on its variant discussed in~\cite{k:hx}.

\subsection{General diffusions}

In the general case, our strategy is to link the regular shift-invariant diffusion $(X(t), Y(t))$ with the harmonic extension problem
\formula[eq:pde]{
 & \begin{cases}
 \op u(x, y) = 0 & \text{for $(x, y) \in \R \times (0, R)$,} \\
 u(x, 0) = f(x) & \text{for $x \in \R$,}
 \end{cases}
}
with an additional Dirichlet condition $u(x, R) = 0$ if $R \in (0, \infty)$. Here $f$ is a given function on the boundary, $u$ is assumed to be appropriately regular, and $L$ is the generator of $(X(t), Y(t))$. Formally, we have
\formula[eq:op]{
 \op & = \tfrac{1}{2} \bigl(a(dy) + (b(y))^2\bigr) \partial_{xx} + b(y) \partial_{xy} + \tfrac{1}{2} \partial_{yy} .
}
We will prove that the infinitesimal generator of $Z(t)$ is the Dirichlet-to-Neumann operator associated to the problem~\eqref{eq:pde}, that is, the operator
\formula[eq:dn]{
 \dn f(x) & = \partial_y u(x, 0) = \lim_{y \to 0^+} \frac{u(x, y) - u(x, 0)}{y} \, .
}

Identification of the generator of the boundary trace of a Markov process with the appropriate Dirichlet-to-Neumann operator is a part of folklore, and it has been proved rigorously in a number of contexts; for references and discussion, see Introduction. However, apparently none of the known results covers the present case. The reasons for this situation are two-fold. First, under rather general assumptions on the coefficients $a(dy)$ and $b(y)$ it is already problematic to give a rigorous meaning to~\eqref{eq:op}. This is carefully carried out in~\cite{k:hx} (in the setting of square integrable functions), but the details are quite technical, and we will not discuss them here. Second, it is not straightforward to identify the operator $\op$ given by~\eqref{eq:op} with the infinitesimal generator of $(X(t), Y(t))$, and the corresponding Dirichlet-to-Neumann operator $\dn$ with the infinitesimal generator of $Z(t)$.

Therefore, we will take a different approach, involving Fourier transform with respect to the variable $x$. This will lead to the one-dimensional Brownian motion $Y(t)$ with a multiplicative functional $\hat{X}(t)$. Using the Itô--Tanaka formula (Corollary~\ref{cor:it}), we will link the pair $(\hat{X}(t), Y(t))$ with an ordinary differential equation, rather than a partial differential equation~\eqref{eq:pde} associated to $(X(t), Y(t))$.

Before we proceed, let us state a variant of the main result of~\cite{ek}, taken almost verbatim from~\cite{k:hx}. We introduce two minor modifications: first, we write $a(dy) + (b(y))^2 dy$ for the measure denoted by $a(dy)$ in~\cite{k:hx}, and so we simply assume that $a(dy)$ is non-negative rather than $a(dy) \ge (b(y))^2 dy$; second, we use the more intuitive expression $-\ph'(0^+) + a(\{0\}) \xi^2$ instead of $-\ph'(0)$ used in~\cite{k:hx} (see clarifications in Section~1.4 therein). Recall that a Rogers function is essentially the characteristic exponent of a Lévy process with completely monotone jumps (restricted to $(0, \infty)$, and then extended to a holomorphic function in the right complex half-plane).

\begin{theorem}[Theorem~2.1 in~\cite{k:hx}]
\label{thm:ode}
\begin{enumerate}[label={\textnormal{(\alph*)}}]
\item\label{thm:ode:a}
Suppose that $R \in (0, \infty]$, $a(dy)$ is a locally finite non-negative measure on $[0, R)$, and $b(y)$ is a locally square-integrable function on $[0, R)$. For every $\xi > 0$ there is a unique function $\ph$ with the following properties: $\ph$ is locally absolutely continuous on $[0, R)$; $\ph'$ is equal almost everywhere in $[0, R)$ to a function with locally bounded variation on $[0, R)$, so that the distributional derivative $\ph''$ corresponds to a locally finite measure in $[0, R)$; $\ph(0) = 1$, $\ph$ is bounded on $[0, R)$, and if $R \in (0, \infty)$, then additionally $\ph(R^-) = 0$; finally, we have
\formula[eq:ode]{
 \tfrac{1}{2} \ph''(dy) & = \tfrac{1}{2} \xi^2 \ph(y) a(dy) + \tfrac{1}{2} \xi^2 (b(y))^2 dy - i \xi \ph'(y) b(y) dy
}
in $(0, R)$, in the sense of distributions. Additionally, $|\ph(y)|^2$ is decreasing and convex on $[0, R)$, $|\ph'(y)|$ is decreasing on $[0, R)$, and $|\ph(y) b(y)|^2 + |\ph'(y)|^2$ is integrable over $[0, R)$.
\item\label{thm:ode:b}
The function $\psi(\xi)$, defined on $(0, \infty)$ by the formula
\formula{
 \psi(\xi) & = -\tfrac{1}{2} \ph'(0^+) + \tfrac{1}{2} a(\{0\}) \xi^2 ,
}
extends to a Rogers function.
\item\label{thm:ode:c}
Every Rogers function $\psi$ can be represented as above in a unique way.
\end{enumerate}
\end{theorem}

Recall that $(X(t), Y(t))$ is a function of two independent Brownian motions, $W(t)$ (on $\R$) and $Y(t)$ (on $[0, R)$). Let $\ex_W$ denote the integration with respect to the law of $W(t)$ alone, with $X(0) = 0$; equivalently, $\ex_W$ is the conditional expectation, with the condition given by the path of $Y(t)$. Similarly, we denote by $\ex_Y^y$ the expectation for the sole process $Y(t)$, with $Y(0) = 0$; $\ex_Y^y$ will only be applied to expressions that no longer depend on $W(t)$.

\begin{proof}[Proof of Theorem~\ref{thm:main}]
Suppose that $(X(t), Y(t))$ is a regular shift-invariant diffusion in $\R \times [0, R)$, where $R \in (0, \infty]$. We use the notation $L_y(t)$, $\dot{Y}(t)$, $A(t)$, $B(t)$, $W(t)$, $a(dy)$ and $b(y)$ introduced in Section~\ref{sec:two:red}. The trace left by $(X(t), Y(t))$ on the boundary is denoted by $Z(u)$; that is, $Z(u) = X(L_0^{-1}(u))$. In particular, by~\eqref{eq:xt},
\formula[eq:xt:zt]{
\begin{aligned}
 X(t) & = X(0) + W(A(t)) + B(t) , \\
 Z(u) & = Z(0) + W(A(L_0^{-1}(u))) + B(L_0^{-1}(u)) ,
\end{aligned}
}
where $t < T_R$ and $u < L_0(T_R)$, and where $A(t)$ and $B(t)$ are given by~\eqref{eq:at} and~\eqref{eq:bt}:
\formula[eq:at:bt]{
 A(t) & = \int_{[0, R)} L_y(t) a(dy) , & B(t) & = \int_0^t b(Y(s)) d\dot{Y}(s) .
}
The corresponding quadratic variation processes are given by:
\formula{
 \qv{A}(t) & = 0 , & \qv{B}(t) & = \int_0^t (b(Y(s)))^2 ds .
}
We fix $\xi > 0$, and we define
\formula[eq:hx:hz]{
 \hat{X}(t) & = \ind_{\{t < T_R\}} \ex_W e^{i \xi X(t)} , & \qquad \hat{Z}(u) & = \ind_{\{u < L_0(T_R)\}} \ex_W e^{i \xi Z(u)} .
}
Clearly,
\formula[eq:ehz]{
 \ex_Y^0 \hat{Z}(u) & = \ex^{(0, 0)} (\ind_{\{u < L_0(T_R)\}} e^{i \xi Z(u)}) = e^{-u \Psi(\xi)} ,
}
where $\Psi(\xi)$ is the characteristic (Lévy--Khintchine) exponent of $Z(u)$. Our goal is to prove that $Z(u)$ is a Lévy process with completely monotone jumps, or, equivalently, that $\Psi(\xi)$ extends to a Rogers function. Let us therefore study $\hat{X}(t)$ and $\hat{Z}(u) = \hat{X}(L_0^{-1}(u))$ in more detail.

All the following equalities are understood to hold with probability $1$ when $t < T_R$. By~\eqref{eq:xt:zt} and our assumption that $X(0) = 0$, we have
\formula{
 \hat{X}(t) & = e^{i \xi B(t)} \ex_W e^{i \xi W(A(t))} = e^{i \xi B(t) - \xi^2 A(t) / 2} ,
}
with $A(t)$ and $B(t)$ given by~\eqref{eq:at:bt}. Note that $|\hat{X}(t)| \le 1$. By Itô's lemma,
\formula[eq:xh]{
\begin{aligned}
 \hat{X}(t) - \hat{X}(0) = i \xi \int_0^t \hat{X}(s) dB(s) & - \frac{\xi^2}{2} \int_0^t \hat{X}(s) dA(s) - \frac{\xi^2}{2} \int_0^t \hat{X}(s) d \langle B \rangle (s) \\
 = i \xi \int_0^t \hat{X}(s) b(Y(s)) dY(s) & - \frac{\xi^2}{2} \int_0^t \hat{X}(s) dA(s) - \frac{\xi^2}{2} \int_0^t \hat{X}(s) (b(Y(s))^2 ds .
\end{aligned}
}
Let $\ph(y)$ be the function described in Theorem~\ref{thm:ode}\ref{thm:ode:a}, let $\psi(\xi) = -\phi'(0) + \tfrac{1}{2} a(\{0\}) \xi^2$ as in Theorem~\ref{thm:ode}\ref{thm:ode:b}, and let $\Phi(t) = \ph(Y(t))$. By the Itô--Tanaka formula (Corollary~\ref{cor:it}), we have
\formula{
 \Phi(t) - \Phi(0) & = \int_0^t \ph'(Y(s)) dY(s) + \frac{1}{2} \int_{(0, R)} L_y(t) \ph''(dy) .
}
Since $\ph$ is a solution of~\eqref{eq:ode}, we find that
\formula{
 \Phi(t) - \Phi(0) & = \int_0^t \ph'(Y(s)) dY(s) + \frac{\xi^2}{2} \int_{(0, R)} L_y(t) \ph(y) a(dy) \\
 & \qquad + \frac{\xi^2}{2} \int_{(0, R)} L_y(t) \ph(y) (b(y))^2 dy - i \xi \int_{(0, R)} L_y(t) \ph'(y) b(y) dy
}
Let $\dot{Y}(t) = Y(t) - \tfrac{1}{2} L_0(t)$ be the martingale part in the canonical decomposition of the semi-martingale $Y(t)$. Clearly,
\formula{
 \int_0^t \ph'(Y(s)) dY(s) & = \int_0^t \ph'(Y(s)) d\dot{Y}(s) + \tfrac{1}{2} \ph'(0^+) L_0(t) .
}
Since $\tfrac{1}{2} \ph'(0^+) = \tfrac{1}{2} \xi^2 a(\{0\}) - \psi(\xi)$, we find that
\formula{
 \Phi(t) - \Phi(0) & = -\psi(\xi) L_0(t) + \int_0^t \ph'(Y(s)) d\dot{Y}(s) + \frac{\xi^2}{2} \int_{[0, R)} L_y(t) \ph(y) a(dy) \\
 & \qquad + \frac{\xi^2}{2} \int_{(0, R)} L_y(t) \ph(y) (b(y))^2 dy - i \xi \int_{(0, R)} L_y(t) \ph'(y) b(y) dy
}
By the occupation time formula (Corollary~\ref{cor:paf}) and the definition~\eqref{eq:at:bt} of $A(t)$,
\formula{
\begin{gathered}
 \int_{[0, R)} L_y(t) \ph(y) a(dy) = \int_0^t \ph(Y(s)) dA(s) , \\
 \int_{(0, R)} L_y(t) \ph(y) (b(y))^2 dy = \int_0^t \ph(Y(s)) (b(Y(s)))^2 ds , \\
 \int_{(0, R)} L_y(t) \ph'(y) b(y) dy = \int_0^t \ph'(Y(s)) b(Y(s)) ds .
\end{gathered}
}
It follows that
\formula[eq:phi]{
\begin{aligned}
 \Phi(t) - \Phi(0) & = -\psi(\xi) L_0(t) + \int_0^t \ph'(Y(s)) d\dot{Y}(s) + \frac{\xi^2}{2} \int_0^t \ph(Y(s)) dA(s) \\
 & \qquad + \frac{\xi^2}{2} \int_0^t \ph(Y(s)) (b(Y(s)))^2 ds + \frac{i \xi}{2} \int_0^t \ph'(Y(s)) b(Y(s)) ds .
\end{aligned}
}
Finally, integration by parts formula for the Itô integral implies that
\formula[eq:parts]{
 \hat{X}(t) \Phi(t) - \hat{X}(0) \Phi(0) & = \int_0^t \hat{X}(s) d\Phi(s) + \int_0^t \Phi(s) d\hat{X}(s) + \qv{\hat{X}, \Phi}(t) ,
}
where $\qv{\hat{X}, \Phi}(t)$ is the quadratic co-variation process. We now combine the above expressions in order to get a formula for $\hat{X}(t) \Phi(t)$.

The martingale parts of $\hat{X}(t)$ and $\Phi(t)$ are given by
\formula{
 & i \xi \int_0^t \hat{X}(s) b(Y(s)) d\dot{Y}(s) && \text{and} & \int_0^t \ph'(Y(s)) d\dot{Y}(s) ,
}
respectively (see~\eqref{eq:xh} and~\eqref{eq:phi}). Furthermore, $\dot{Y}(t)$ is the Brownian motion in~$\R$. Thus, the last term in the right-hand side of~\eqref{eq:parts} is equal to
\formula{
 \qv{\hat{X}, \Phi}(t) & = i \xi \int_0^t \hat{X}(s) \ph'(Y(s)) b(Y(s)) ds .
}
The other two terms are
\formula{
 \int_0^t \hat{X}(s) d\Phi(s) & = - \psi(\xi) \int_0^t \hat{X}(s) dL_0(s) + \int_0^t \hat{X}(s) \ph'(Y(s)) d\dot{Y}(s) \\
 & \qquad + \frac{\xi^2}{2} \, \int_0^t \hat{X}(s) \ph(Y(s))(s) dA(s) + \frac{\xi^2}{2} \int_0^t \hat{X}(s) \ph(Y(s)) (b(Y(s)))^2 ds \\
 & \qquad\qquad - i \xi \int_0^t \hat{X}(s) \ph'(Y(s)) b(Y(s)) ds
}
(by~\eqref{eq:phi}), and
\formula{
 \int_0^t \Phi(s) d\hat{X}(s) & = i \xi \int_0^t \hat{X}(s) \ph(Y(s)) b(Y(s)) d\dot{Y}(s) \\
 & \qquad - \frac{\xi^2}{2} \int_0^t \hat{X}(s) \ph(Y(s)) dA(s) - \frac{\xi^2}{2} \int_0^t \hat{X}(s) \ph(Y(s)) (b(Y(s))^2 ds
}
(by~\eqref{eq:xh}). By adding the sides of the above identities and using~\eqref{eq:parts}, we conclude that
\formula[eq:mart:1]{
\begin{aligned}
 \hat{X}(t) \Phi(t) - \hat{X}(0) \Phi(0) & = - \psi(\xi) \int_0^t \hat{X}(s) dL_0(s) \\
 & \hspace*{-2em} + \int_0^t \hat{X}(s) \ph'(Y(s)) d\dot{Y}(s) + i \xi \int_0^t \hat{X}(s) \ph(Y(s)) b(Y(s)) d\dot{Y}(s) 
\end{aligned}
}
when $t < T_R$.

Suppose for the moment that $R \in (0, \infty)$. By assumption, $\Phi(T_R^-) = \ph(R^-) = 0$. We claim that all three integrals in the right-hand side exist for $t = T_R$. Indeed: $|\hat{X}(s)| \le 1$ and $|\ph'(s)| \le -\ph'(0^+)$, and hence
\begin{gather}
 \notag \ex_Y^0 \int_0^{T_R} |\hat{X}(s)| dL_0(s) \le \ex_Y^0 L_0(T_R) < \infty , \\
 \label{eq:ests} \ex_Y^0 \int_0^{T_R} |\hat{X}(s) \ph'(Y(s))|^2 ds \le (\ph'(0^+))^2 \ex_Y^0 T_R < \infty , \\
 \notag \ex_Y^0 \int_0^{T_R} |\hat{X}(s) \Phi(s) b(Y(s))|^2 ds \le R \int_0^R |\ph(y) b(y)|^2 dy < \infty ;
\end{gather}
in the last estimate we used the occupation time formula (Corollary~\ref{cor:paf}), the expression for the mean value of the local time (Proposition~\ref{prop:otd}) and integrability of $|\ph(y) b(y)|^2$ (Theorem~\ref{thm:ode}\ref{thm:ode:a}). Thus, if we agree that $\hat{X}(T_R) \Phi(T_R) = 0$, then~\eqref{eq:mart:1} also holds for $t = T_R$.

Observe also that, by our convention, $X(0) = 0$, and so $\hat{X}(0) = 1$. Furthermore, $\Phi(0) = \ph(Y(0))$, so if $Y(0) = 0$, then $\Phi(0) = 1$. Combining this observation with the above remark, we find that, regardless of whether $R$ is finite or not, we can rewrite~\eqref{eq:mart:1} as
\formula[eq:mart:2]{
\begin{aligned}
 \hat{X}(t \wedge T_R) \Phi(t \wedge T_R) - 1 & = - \psi(\xi) \int_0^{t \wedge T_R} \hat{X}(s) dL_0(s) \\
 & \hspace*{-6em} + \int_0^{t \wedge T_R} \hat{X}(s) \ph'(Y(s)) d\dot{Y}(s) + i \xi \int_0^{t \wedge T_R} \hat{X}(s) \ph(Y(s)) b(Y(s)) d\dot{Y}(s) 
\end{aligned}
}
for every $t \ge 0$, provided that $Y(0) = 0$. Note that
\formula{
 \hat{X}(L_0^{-1}(u)) \Phi(L_0^{-1}(u)) & = \hat{Z}(u) \ph(0) = \hat{Z}(u)
}
if $L_0^{-1}(u) < T_R$, and
\formula{
 \hat{X}(T_R) \Phi(T_R) & = 0 = \hat{Z}(u)
}
if $L_0^{-1}(u) \ge T_R$ by our convention and the definition~\eqref{eq:hx:hz} of $Z(u)$. Therefore
\formula{
 \hat{X}(L_0^{-1}(u) \wedge T_R) \Phi(L_0^{-1}(u) \wedge T_R) & = \hat{Z}(u) .
}
Similarly,
\formula{
 \int_0^{L_0^{-1}(u) \wedge T_R} \hat{X}(s) dL_0(s) & = \int_0^{L_0^{-1}(u)} \ind_{\{s < T_R\}} \hat{X}(s) dL_0(s) = \int_0^u \hat{Z}(v) dv .
}
Thus, by setting $t = L_0^{-1}(u)$ in~\eqref{eq:mart:2}, we find that:
\formula{
 \hat{Z}(u) - 1 & = - \psi(\xi) \int_0^u \hat{Z}(v) dv \\
 & \quad + \int_0^{L_0^{-1}(u) \wedge T_R} \hat{X}(s) \ph'(Y(s)) d\dot{Y}(s) + i \xi \int_0^{L_0^{-1}(u) \wedge T_R} \hat{X}(s) \ph(Y(s)) b(Y(s)) d\dot{Y}(s)
}
for every $u \ge 0$. Taking expectations of both sides (and using estimates~\eqref{eq:ests} to find that the expectation of the two Itô integrals is zero), we end up at
\formula{
 \ex_Y^0 \hat{Z}(u) - 1 & = -\psi(\xi) \int_0^u \ex_Y^0 \hat{Z}(v) dv .
}
The solution of this integral equation is clearly given by
\formula{
 \ex_Y^0 \hat{Z}(u) & = \exp(-u \psi(\xi)) .
}
On the other hand, we have already seen in~\eqref{eq:ehz} that
\formula{
 \ex_Y^0 \hat{Z}(u) & = \exp(-u \Psi(\xi)) .
}
It follows that $\Psi(\xi) = \psi(\xi)$ for $\xi > 0$. By Theorem~\ref{thm:ode}\ref{thm:ode:b} we conclude that $\Psi(\xi)$ extends to a Rogers function, and consequently $\hat{Z}(u)$ is a Lévy process with completely monotone jumps. Conversely, by Theorem~\ref{thm:ode}\ref{thm:ode:c}, every Lévy process with completely monotone jumps can be realised as $\hat{Z}(u)$ in a unique way.
\end{proof}

We note that the random variable $L_0(T_R)$, the life-time of $Z(u)$, is exponentially distributed, with mean $R$; see formula~3.2.3.2 in~\cite{bs}. Therefore, $Z(u)$ is killed with intensity $\gamma = 1 / R$, as mentioned in the introduction.

\begin{proof}[Proof of Corollary~\ref{cor:loc}]
It suffices to observe that the process $Z(u)$ in~\eqref{eq:thm:loc} is given by the same expression as in~\eqref{eq:xt:zt}, with $Z(0) = 0$, and with $W(t)$ denoted by $X(t)$. Thus, the desired result follows from the proof of Theorem~\ref{thm:main}.
\end{proof}

\begin{proof}[Proof of Corollary~\ref{cor:exc}]
Let $Z(t)$ be a Lévy process with completely monotone jumps. We identify $Z(t)$ with the trace of a regular shift-invariant diffusion $(X(t), Y(t))$, as in Theorem~\ref{thm:main}, and we use the notation introduced in Sections~\ref{sec:two:red} and~\ref{sec:exc}. Recall that $L_0(T_R)$ is the life-time of $Z(t)$, and $\gamma = 1 / R$ is the killing rate of $Z(t)$. Furthermore, by $(\delta_u(t), \eps_u(t))$ we denote the excursions of $(X(t), Y(t))$ away from $\R \times \{0\}$:
\formula{
 \delta_u(t) & = X((\alpha_u + t) \wedge \beta_u) - X(\alpha_u), & \eps_u(t) & = Y((\alpha_u + t) \wedge \beta_u) ,
}
where $\alpha_u = L_0^{-1}(u^-)$ and $\beta_u = L_0^{-1}(u)$. The life-time of $(\delta(t), \eps(t))$ is denoted by $\zeta_u = \beta_u - \alpha_u$.

The jumps of the process $Z(t)$ are given by
\formula{
 \Delta Z(u) & = Z(u) - Z(u^-) = X(L^{-1}(u)) - X(L^{-1}(u^-)) = X(\beta_u) - X(\alpha_u) .
}
Therefore, by the Lévy--Itô theorem, the Lévy measure $\nu(dx)$ of the process $Z(u)$ is given by
\formula[eq:nu:sum]{
\begin{aligned}
 \biggl(\int_0^u (1 - e^{-\gamma v}) dv\biggr) \nu(A) & = \ex^{(0,0)} \sum_{v \in [0, u \wedge L_0(T_R))} \ind_A(\Delta Z(v)) \\
 & = \ex^{(0,0)} \sum_{v \in [0, u \wedge L_0(T_R))} \ind_A(X(\beta_v) - X(\alpha_v)) .
\end{aligned}
}

Thus, we can rewrite~\eqref{eq:nu:sum} as
\formula{
 \biggl(\int_0^u (1 - e^{-\gamma v}) dv\biggr) \nu(A) & = \ex^{(0,0)} \sum_{v \in [0, u \wedge L_0(T_R))} \ind_A(\delta_v(\zeta_v)) .
}
By the exit system formula~\eqref{eq:exit},
\formula{
 \biggl(\int_0^u (1 - e^{-\gamma v}) dv\biggr) \nu(A) & = \ex^{(0,0)} \int_0^{u \wedge L_0(T_R)} \exit\bigl(\ind_{E_R}(\eps) \ind_A(\delta(\zeta))\bigr) du \\
 & = \ex^{(0,0)} \bigl(u \wedge L_0(T_R)\bigr) \exit\bigl(\ind_{E_R}(\eps) \ind_A(\delta(\zeta))\bigr) \\
 & = \biggl(\int_0^u (1 - e^{-\gamma v}) dv\biggr) \exit(\eps \in E_R, \, \delta(\zeta) \in A) ,
}
where $\exit(d\delta, d\eps)$ is the corresponding exit system, and $E_R$ is the class of excursions $\eps(t)$ of $Y(t)$ which never reach $R$. Thus,
\formula{
 \nu(A) & = \exit(\eps \in E_R; \delta(\zeta) \in A) .
}
It remains to express the exit system $\exit$ of $(X(t), Y(t))$ in terms of the excursion measure $\exc(d\eps)$ of $Y(t)$.

By the representation~\eqref{eq:xt:zt},
\formula{
 X(t) & = W(A(t)) + B(t) ,
}
where $W(t)$ and $Y(t)$ are independent Wiener processes (on $\R$ and $[0, R)$, respectively), and by~\eqref{eq:at:bt},
\formula{
 A(t) & = \int_{[0, R)} L_y(t) a(dy) , & B(t) & = \int_0^t b(Y(s)) d\dot{Y}(s) .
}
Therefore, for $u < L_0(T_R)$,
\formula{
 \delta_u(t) & = X((\alpha_u + t) \wedge \beta_u) - X(\alpha_u) \\
 & = \bigl(W(A((\alpha_u + t) \wedge \beta_u)) - W(A(\alpha_u))\bigr) + \bigl(B((\alpha_u + t) \wedge \beta_u) - B(\alpha_u)\bigr) .
}
The intervals $(\alpha_u, \beta_u)$ are pairwise disjoint, and their distribution is independent from the Brownian motion $W(t)$. It follows that to each excursion $\eps_u(t)$ of the process $Y(t)$ we can associate the stopped Brownian motion
\formula{
 W_u(t) & = W((\alpha_u + t) \wedge \beta_u) - W(\alpha_u) ,
}
and, conditionally on the path of $Y(t)$, $W_u(t)$ are independent stopped Brownian motions in $\R$. As a consequence, we can represent the exit system $\exit(d\delta, d\eps)$ in a similar way. More precisely, suppose that $\eps(t)$ has law given by the excursion measure $\exc(d\eps)$ of $Y(t)$, $W(t)$ is the Brownian motion independent from $\eps(t)$, and
\formula{
 \delta(t) & = W(A_\eps(t \wedge \zeta)) + B_\eps(t \wedge \zeta) ,
}
where
\formula{
 A_\eps(t) & = \int_{[0, R)} L_{\eps,y}(t) a(dy) , & B_\eps(t) & = \int_0^t b(\eps(s)) d\eps(s) ,
}
and $L_{\eps,y}(t)$ is the local time of $\eps(t)$ at $y$. Then the joint law of $(\delta(t), \eps(t))$ (restricted to the event $\eps \in E$) is equal to $\exit(d\delta, d\eps)$ (restricted to the event $\eps \in E$).

We conclude that
\formula[eq:nu:exc]{
 \nu(A) & = \exit(\eps \in E, \delta(\zeta) \in A) = \eta\Bigl(E; \, \pr_W \bigl(W(A_\eps(\zeta)) + B_\eps(\zeta) \in A\bigr)\Bigr) ,
}
as desired.

We have thus proved that every Lévy measure can be represented as in~\eqref{eq:nu:exc}. Conversely, any parameters $R$, $a(dy)$ and $b(y)$ correspond to a unique Lêvy process with completely monotone jumps, and hence they determine some Lévy measure $\nu(dx)$. Finally, if we suppose that $R = \infty$ and $a(\{0\}) = 0$, then the corresponding Lévy process $Z(t)$ has infinite life-time and no Gaussian part. Consequently, the process $Z(t)$ is uniquely determined by its Lévy measure $\nu(dx)$ and the drift. The drift can be changed arbitrarily by adding a constant function to $b(y)$, and so the Lévy measure determines $a(dy)$ uniquely, and $b(y)$ up to addition by a constant.
\end{proof}

%
%                            ---------- o ----------
%

\section{Examples}
\label{sec:ex}

We conclude the article with a number of explicit pairs of Lévy processes with completely monotone jumps and the corresponding regular shift-invariant diffusions. These examples correspond to Section~4 in~\cite{k:hx}, where all technical details omitted here are given. We illustrate each example with a figure, which shows essential features of the corresponding diffusion, together with a sample path.

Assume for the moment that $a(dy)$ has a continuous density function $a(y)$, and $b(y)$ is a continuous function. Then $(X(t), Y(t))$ is a two-dimensional Itô diffusion, with diffusion coefficient matrix
\formula{
 \frac{1}{2} \begin{pmatrix} a(y) + (b(y))^2 & b(y) \\ b(y) & 1 \end{pmatrix} .
}
Let $\vec{v}_1 = (\sqrt{a(y)}, 0)$ and $\vec{v}_2 = (b(y), 1)$. Then the above matrix can be written as $\vec{v}_1 \otimes \vec{v}_1 + \vec{v}_2 \otimes \vec{v}_2$, and hence the diffusion $(X(t), Y(t))$ \emph{locally} behaves as $W_1(t) \vec{v}_1 + W_2(t) \vec{v}_2$, where $W_1(t)$ and $W_2(t)$ are two independent one-dimensional Brownian motions.

Suppose additionally that $b(y)$ is continuously differentiable, and write $B(y) = \int_0^y b(s) ds$. Consider the following curvilinear coordinates: $\tilde{x} = x + B(y)$, $\tilde{y} = y$. Since $\nabla \tilde{x} = (1, b(y))$ and $\nabla \tilde{y} = (0, 1)$, we have $\vec{v}_2 = \nabla \tilde{x}$ and $\vec{v_1} = \sqrt{a(y)} \nabla \tilde{y}$. Thus, in coordinates $\tilde{x}, \tilde{y}$, the diffusion coefficient matrix of $(X(t), Y(t))$ becomes diagonal:
\formula{
 \frac{1}{2} \begin{pmatrix} a(y) & 0 \\ 0 & 1 \end{pmatrix} .
}
On the other hand, in coordinates $(\tilde{x}, \tilde{y})$ the diffusion $(X(t), Y(t))$ is no longer a martingale: it gets a horizontal drift $\vec{w} = (b'(y), 0)$. For further details in the analytical setting, we refer to Section~1.3.1 in~\cite{k:hx}.

We can summarise the above two paragraphs as follows.
\begin{itemize}
\item The two vectors $\vec{v}_1 = (\sqrt{a(y)}, 0)$ and $\vec{v}_2 = (b(y), 1)$ indicate two independent \emph{diffusion directions} of $(X(t), Y(t))$. Length of these vectors correspond to the speed of the diffusion in the given direction.
\item If $a(y) = 0$, then $\vec{v}_1 = 0$, and the diffusion $(X(t), Y(t))$ is \emph{degenerate}. In this case it aligns to the \emph{profile lines} with constant $\tilde{x} = x + B(y)$. Indeed, these lines are integral curves of the vector field $\vec{v_1}$. The diffusion $(X(t), Y(t))$ moves between different profile lines with speed given by the \emph{apparent drift} $\vec{w} = (b'(y), 0)$ (which becomes a true drift coefficient of $(X(t), Y(t))$ in the curvilinear coordinates $\tilde{x} = x + B(y)$, $\tilde{y} = y$).
\item If $a(y) > 0$, then the motion of the diffusion $(X(t), Y(t))$ between different profile lines is a diffusion, with diffusion coefficient $a(y)$ and drift term $b'(y)$.
\item Further motion between profile lines is caused by oblique reflection at the boundary. More precisely, if the limit $b(0^+)$ is finite, then, \emph{locally} near the horizontal axis, $X(t)$ behaves as $b(0^+) \dot{Y}(t) = b(0^+) Y(t) - b(0^+) L_0(t)$. Thus, the diffusion $(X(t), Y(t))$ additionally moves between different profile lines according to the local time of $Y(t)$ at $0$, with relative speed given by $-b(0^+)$. Equivalently, $-b(0^+)$ is the drift coefficient of the trace $Z(u)$. Note that $-b(0^+)$ is also the tangent of the angle between profile lines and the horizontal axis.
\end{itemize}
In each of the figures in this section we show:
\begin{itemize}
\item a sample path of the diffusion $(X(t), Y(t))$ (black line);
\item a corresponding path of the boundary trace $Z(t)$ (blue dots);
\item diffusion directions (red two-sided arrows);
\item apparent drift (magenta arrows);
\item profile lines (grey lines).
\end{itemize}

\afterpage{
\begin{landscape}
\begin{figure}
\centering%\scriptsize
\begin{tabular}{cc}
\includegraphics[width=11cm]{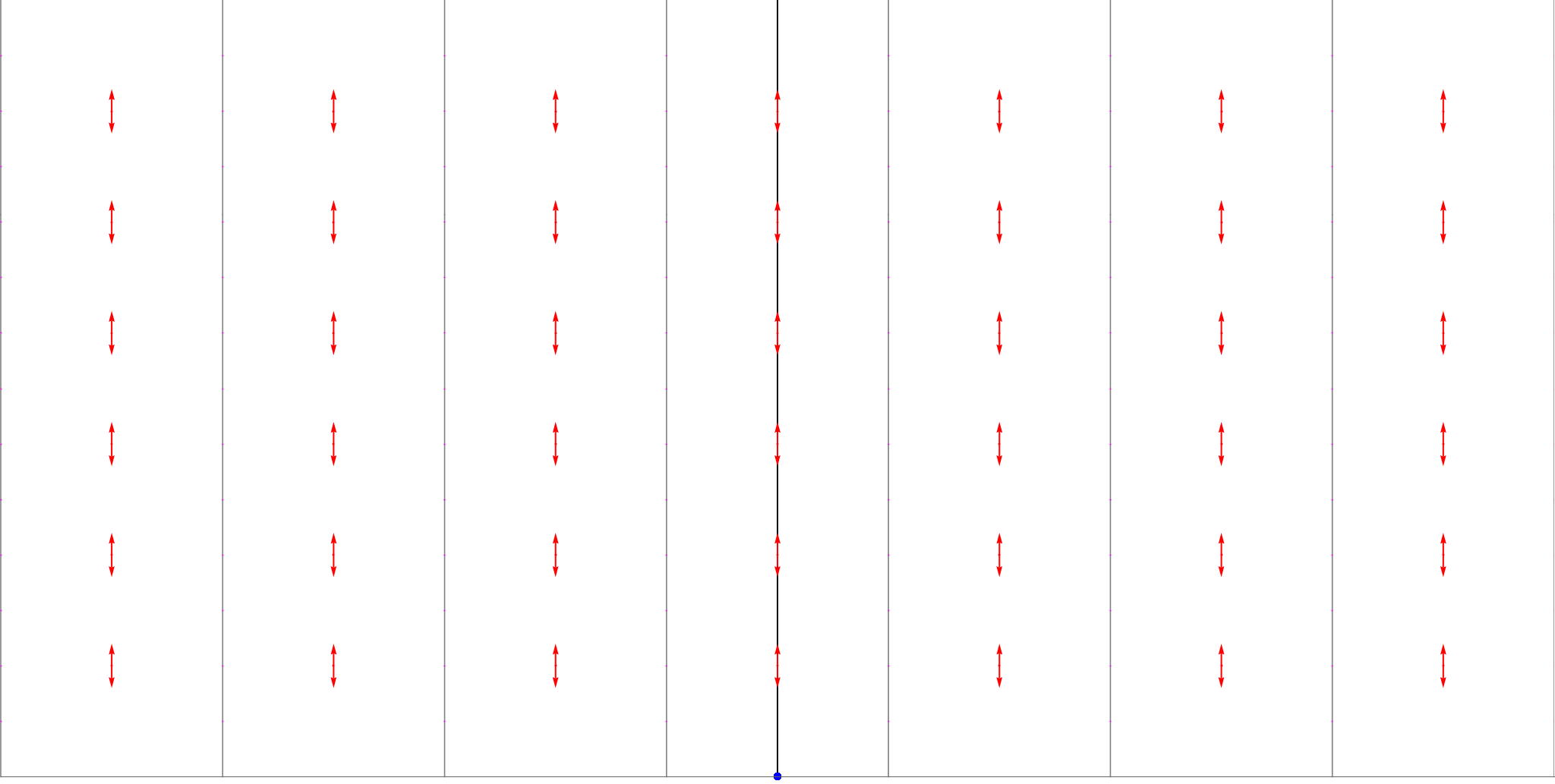}
&
\includegraphics[width=11cm]{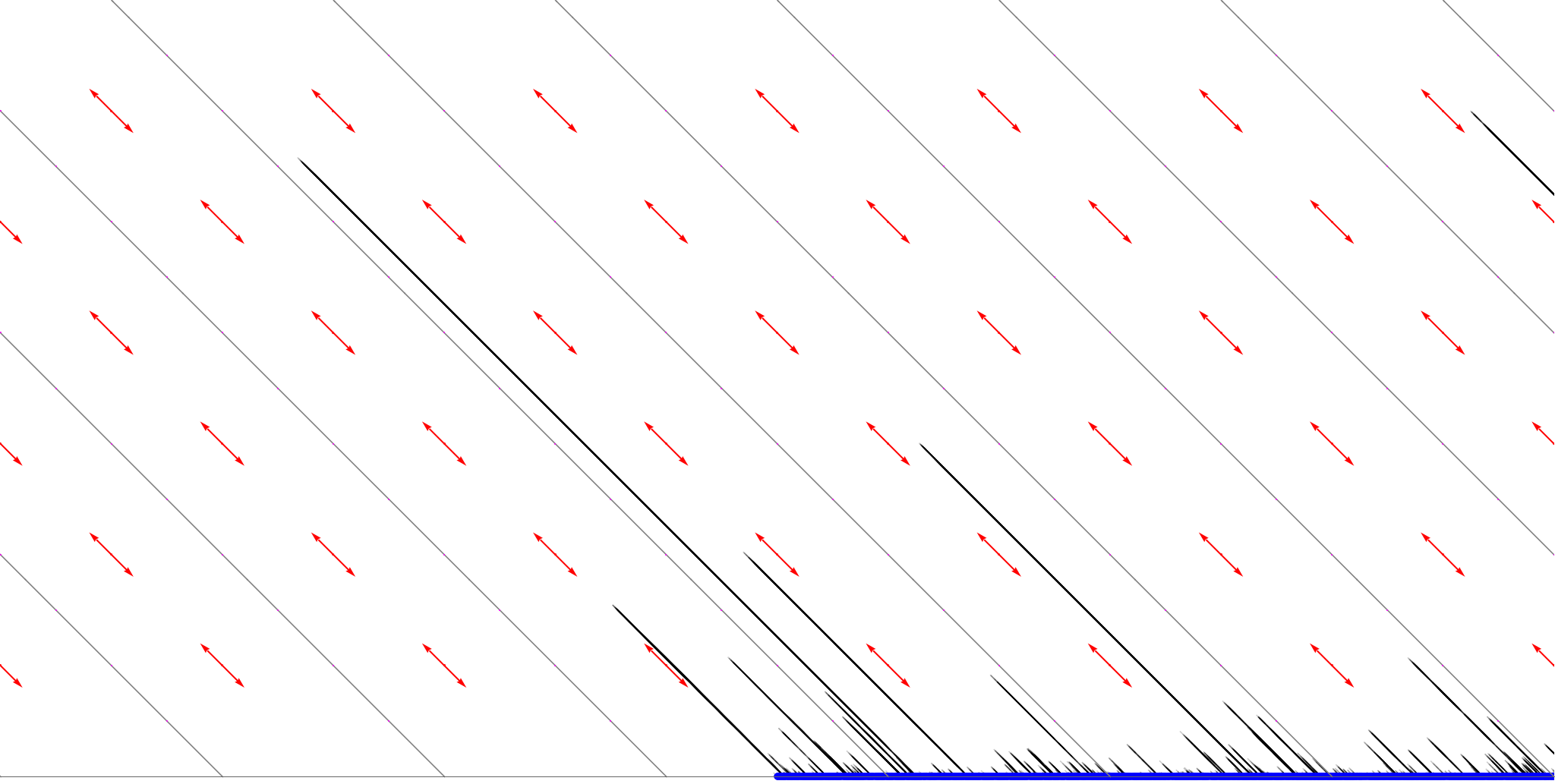}
\\
(a) & (b)
\end{tabular}
\caption{A degenerate diffusion, the trace of which is (a)~a constant process, (b)~a pure drift.}
\label{fig:trivial}
\vspace*{1cm}
\end{figure}

\begin{figure}
\centering%\scriptsize
\begin{tabular}{cc}
\includegraphics[width=11cm]{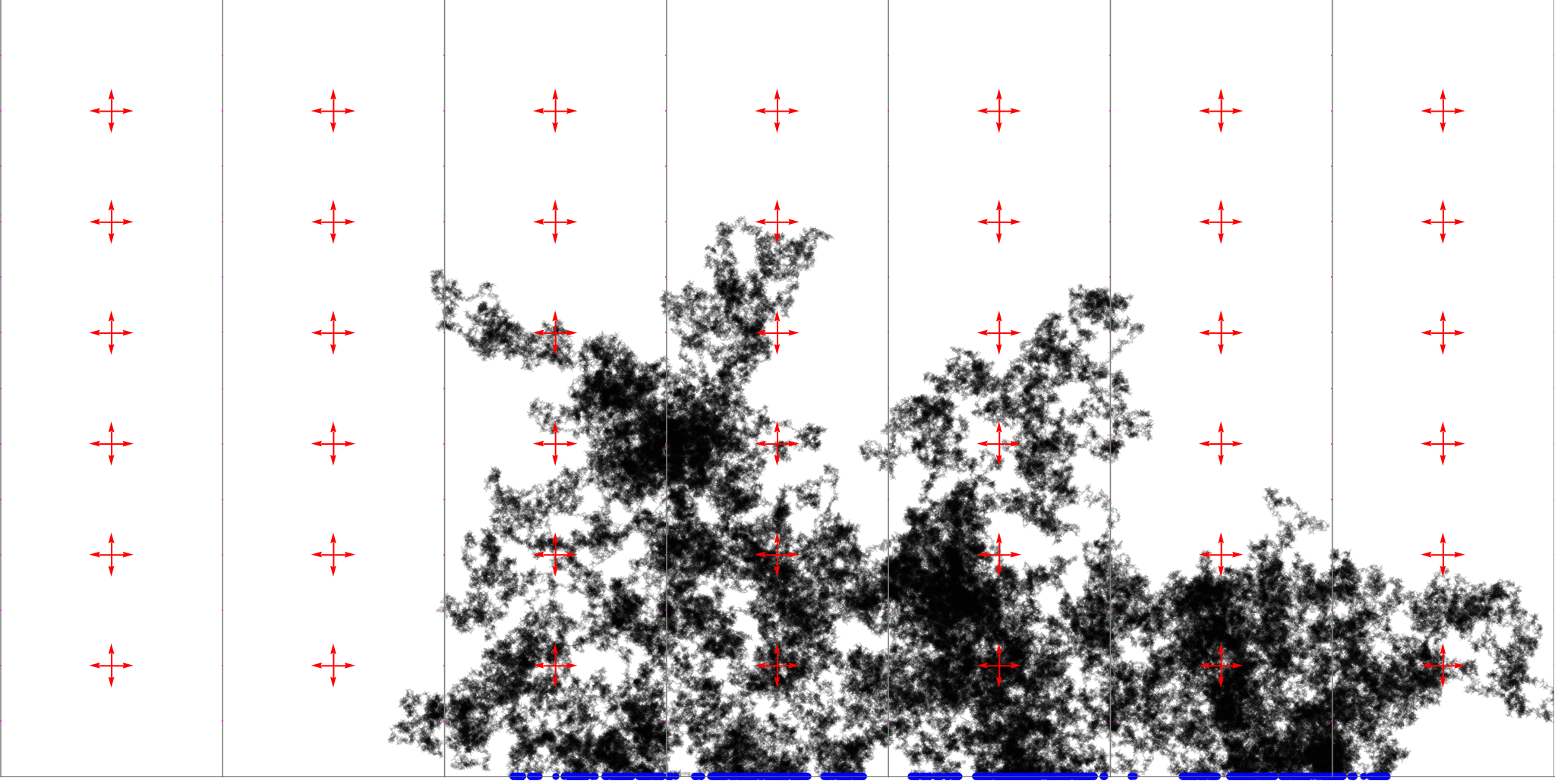}
&
\includegraphics[width=11cm]{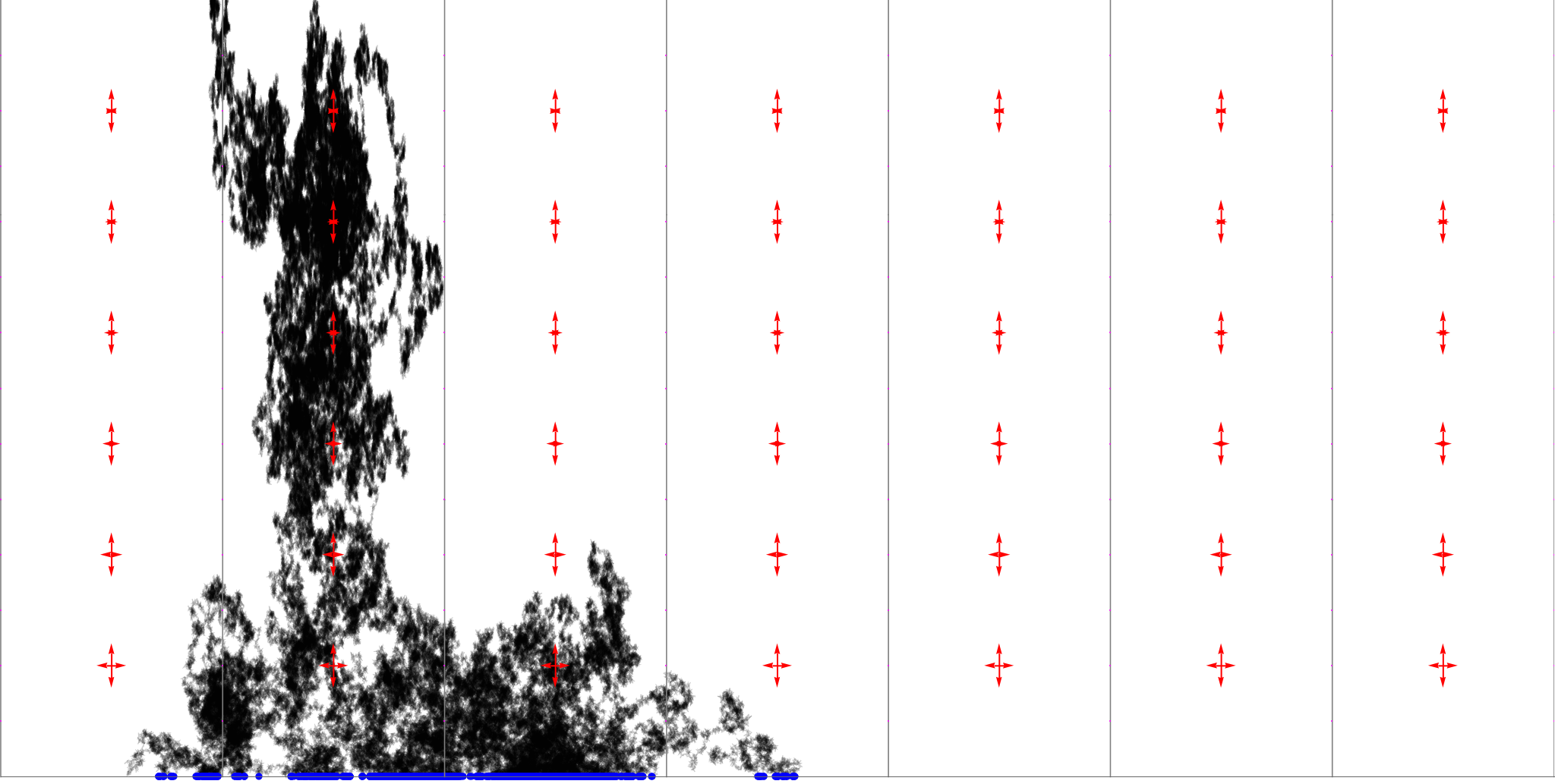}
\\
(a) & (b)
\end{tabular}
\caption{A diffusion, the trace of which is (a) the Cauchy process, (b) the relativistic Cauchy process.}
\label{fig:brown}
\end{figure}

\begin{figure}
\centering%\scriptsize
\begin{tabular}{cc}
\includegraphics[width=11cm]{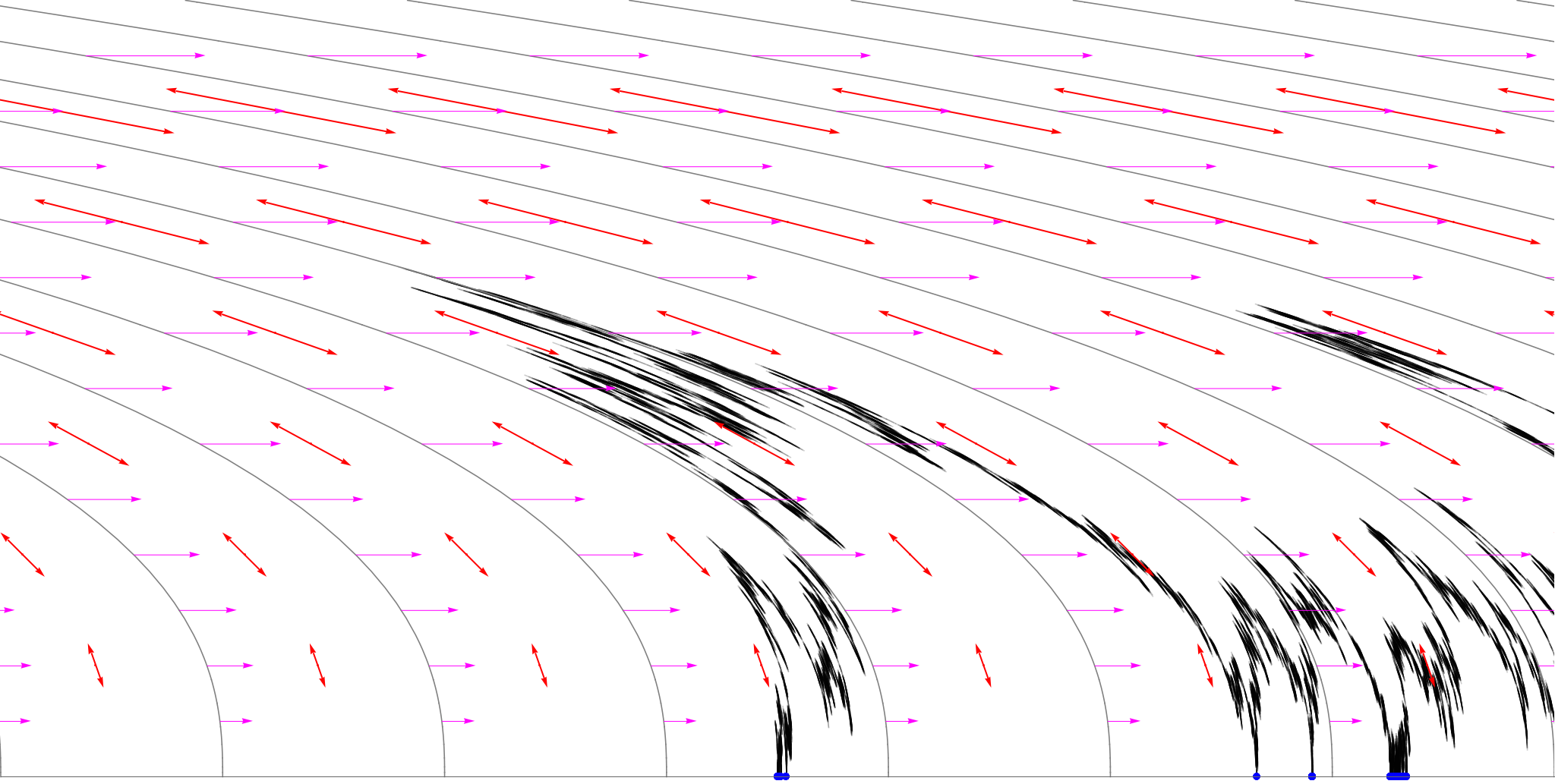}
&
\includegraphics[width=11cm]{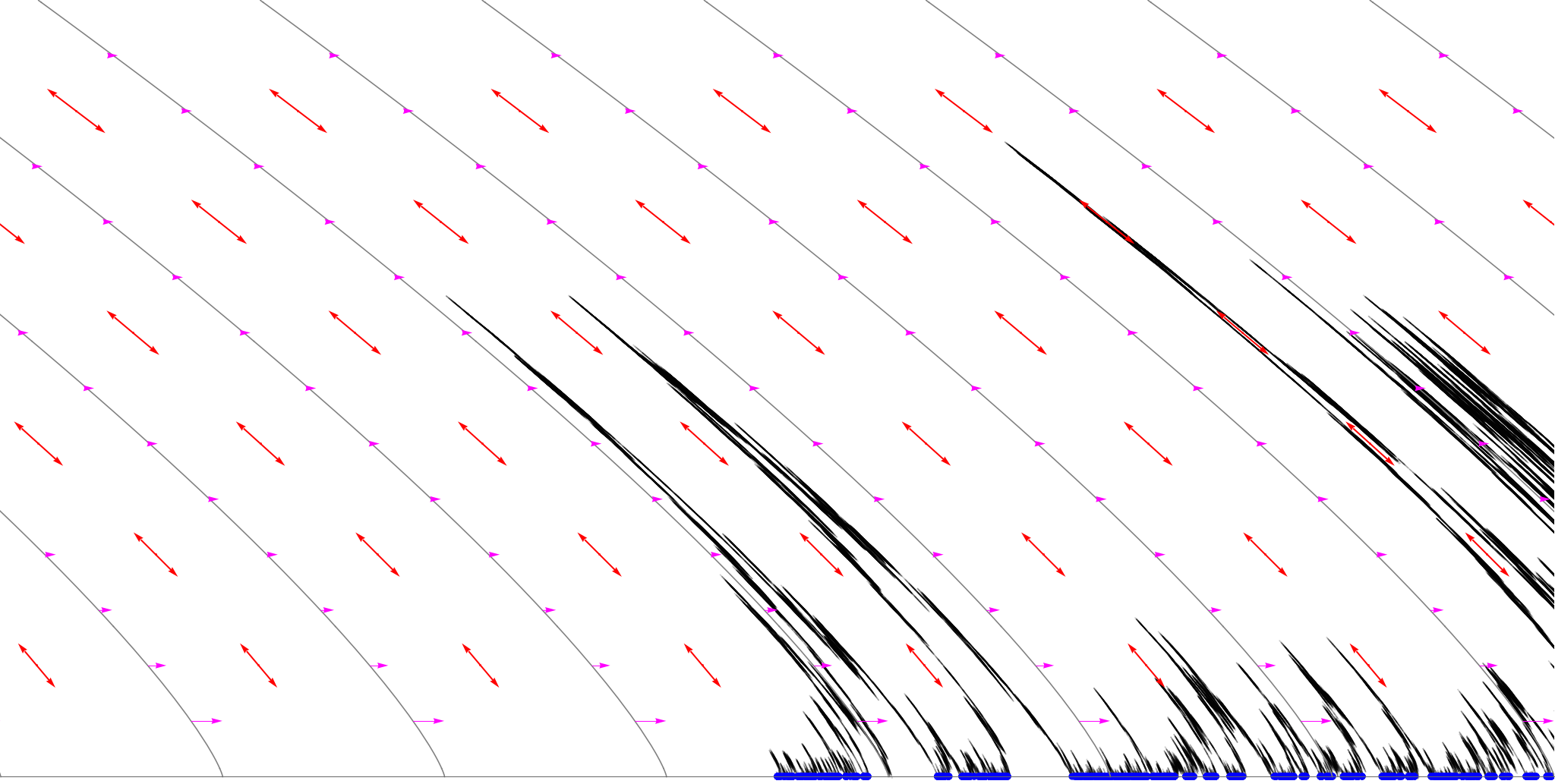}
\\
(a) & (b)
\end{tabular}
\caption{A degenerate diffusion, the trace of which is a stable subordinator with index (a)~$0.4$, (b)~$0.8$.}
\label{fig:sub}
\vspace*{1cm}
\end{figure}

\begin{figure}
\centering%\scriptsize
\begin{tabular}{cc}
\includegraphics[width=11cm]{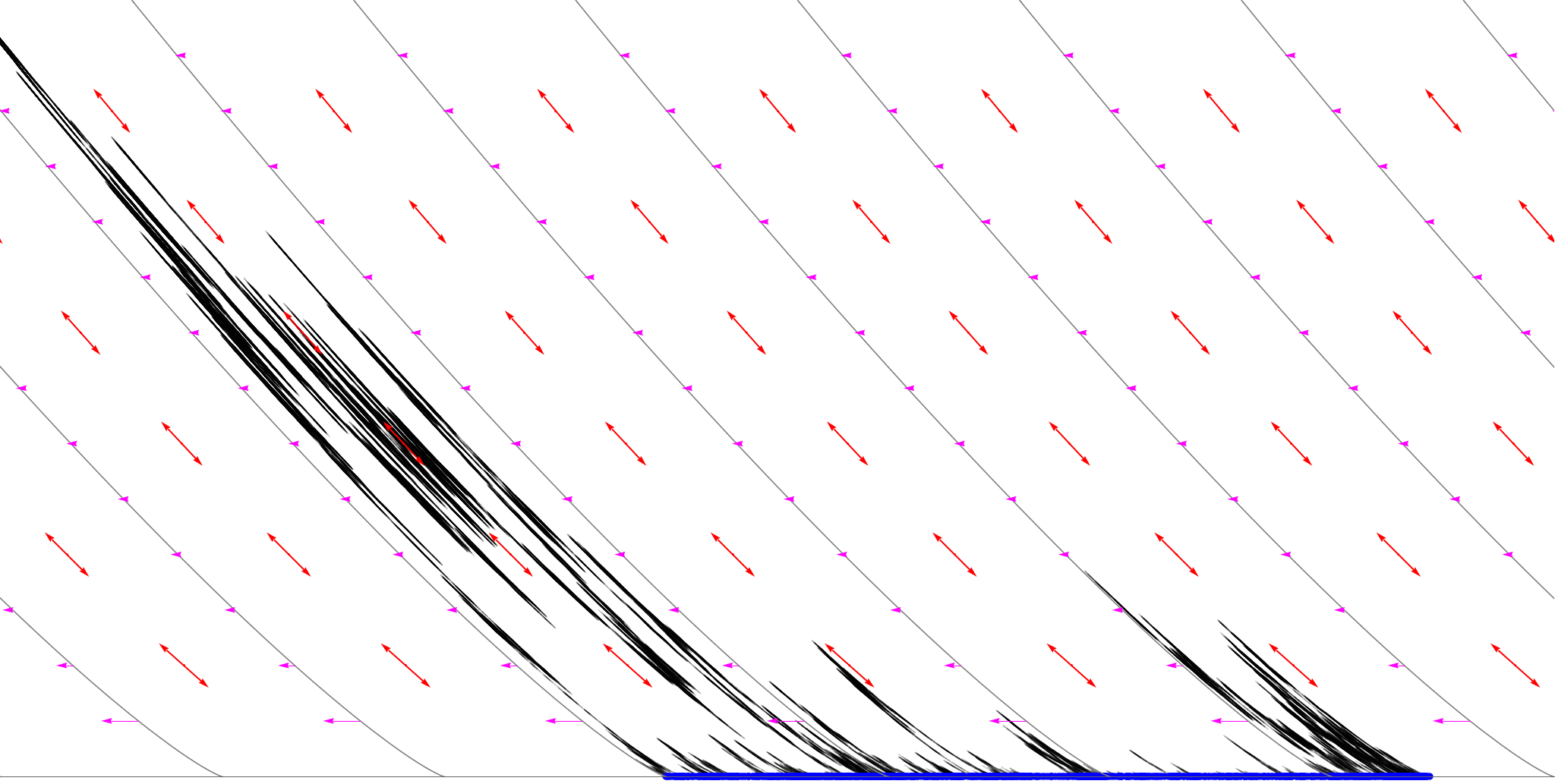}
&
\includegraphics[width=11cm]{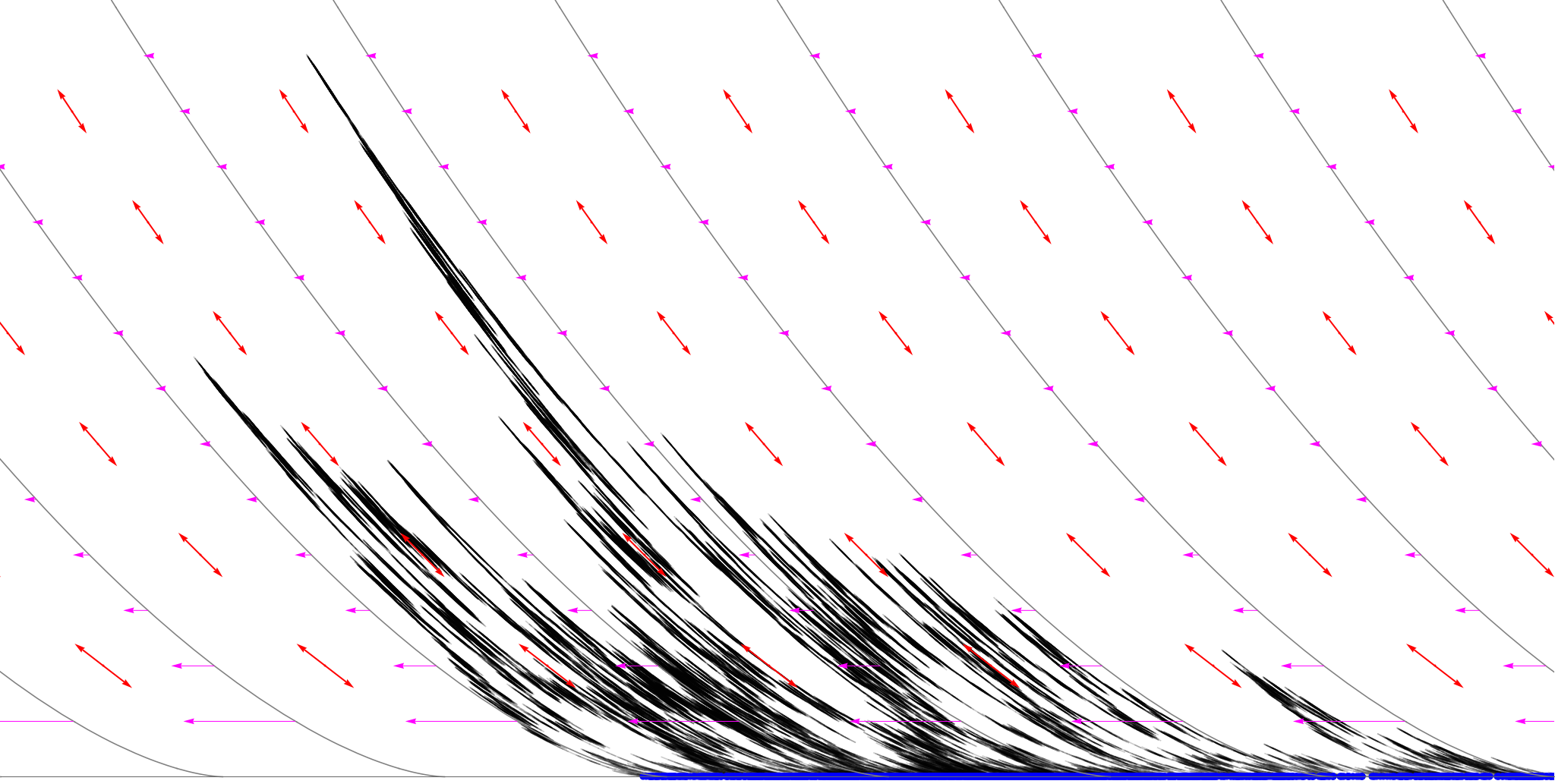}
\\
(a) & (b)
\end{tabular}
\caption{A degenerate diffusion, the trace of which is a one-sided stable process with index (a)~$1.2$, (b)~$1.6$.}
\label{fig:onesided}
\end{figure}

\begin{figure}
\centering%\scriptsize
\begin{tabular}{cc}
\includegraphics[width=11cm]{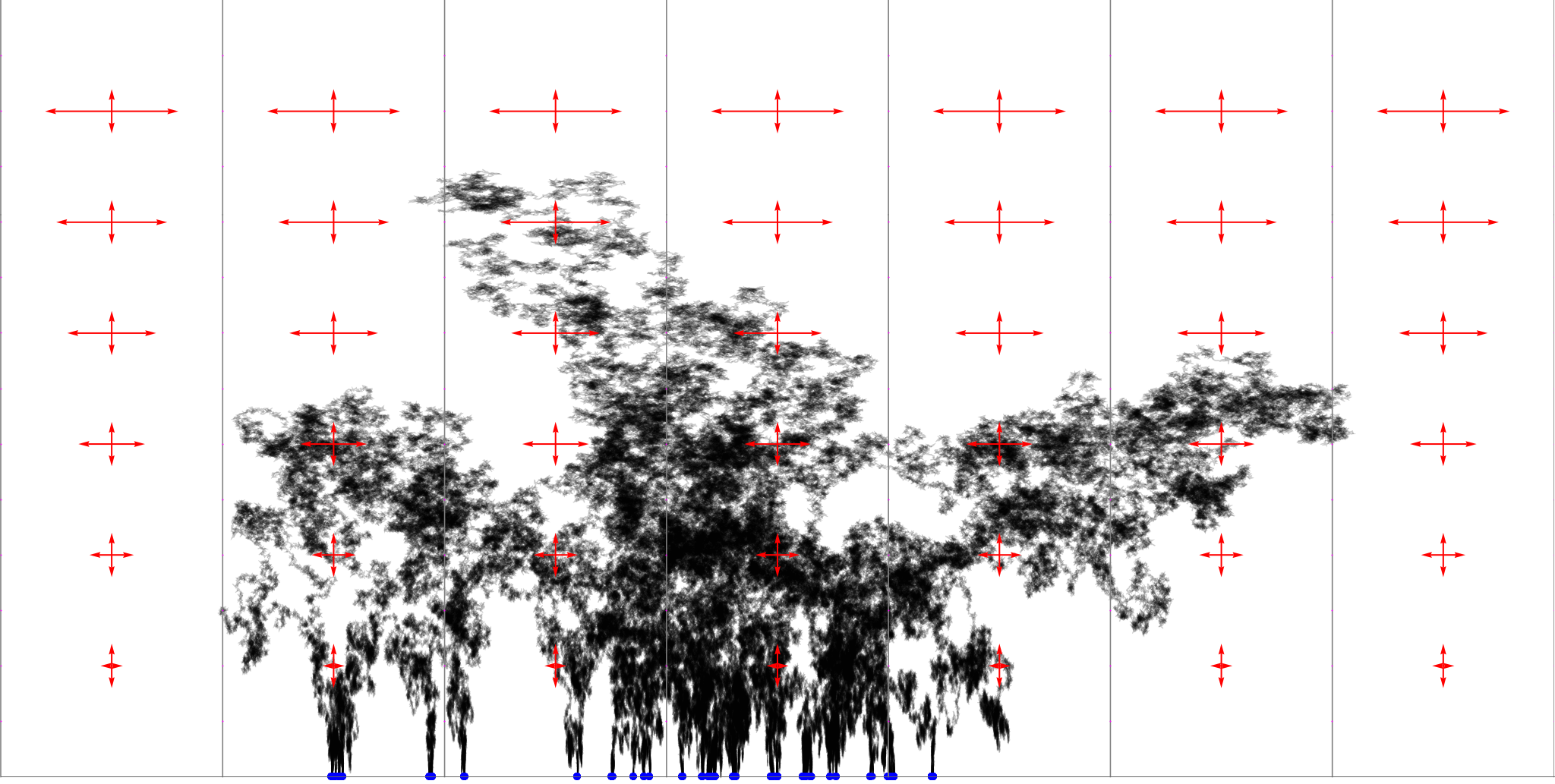}
&
\includegraphics[width=11cm]{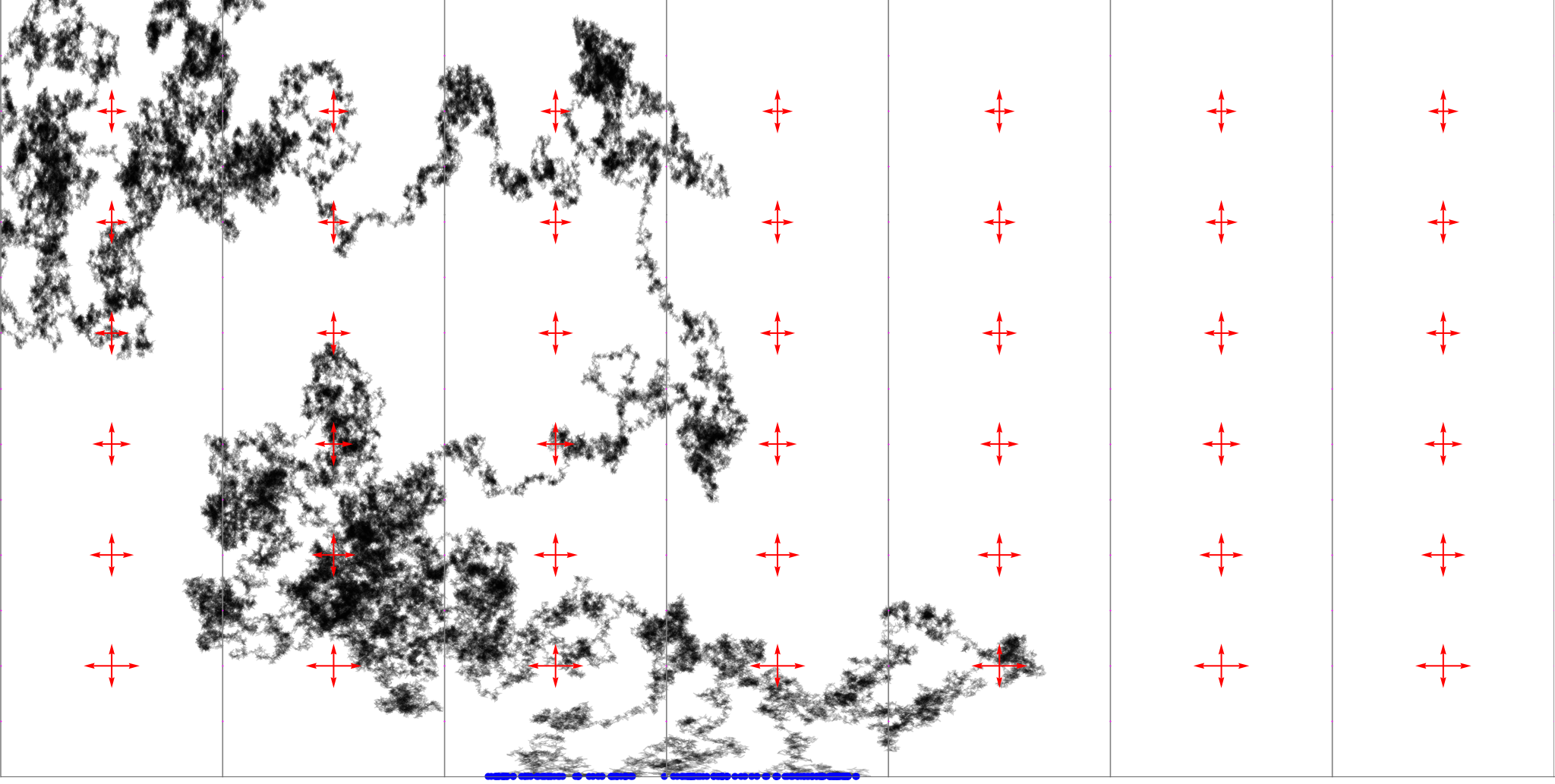}
\\
(a) & (b)
\end{tabular}
\caption{A diffusion, the trace of which is a symmetric stable process with index (a)~$0.5$, (b)~$1.5$.}
\label{fig:sym}
\vspace*{1cm}
\end{figure}

\begin{figure}
\centering%\scriptsize
\begin{tabular}{cc}
\includegraphics[width=11cm]{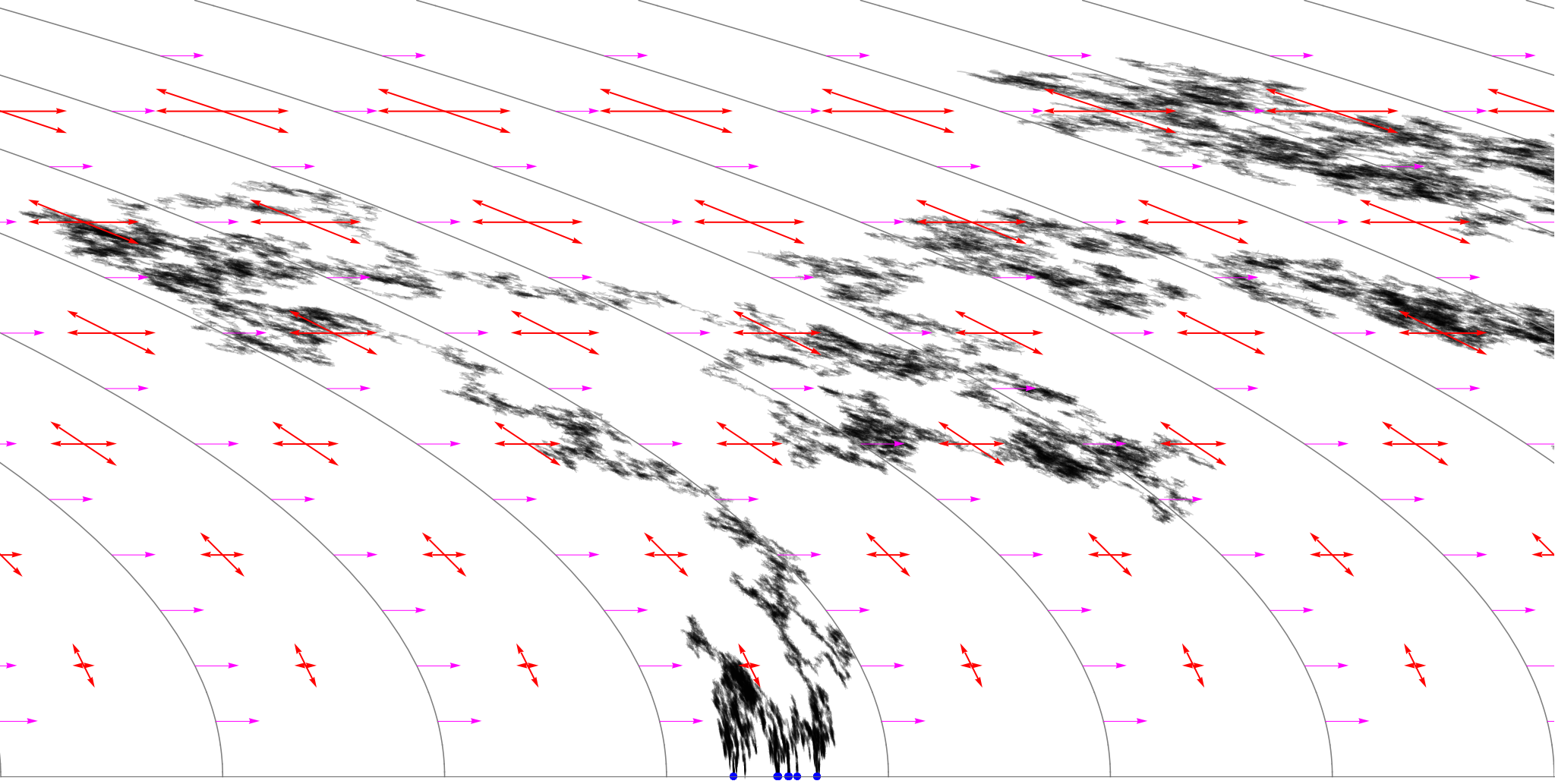}
&
\includegraphics[width=11cm]{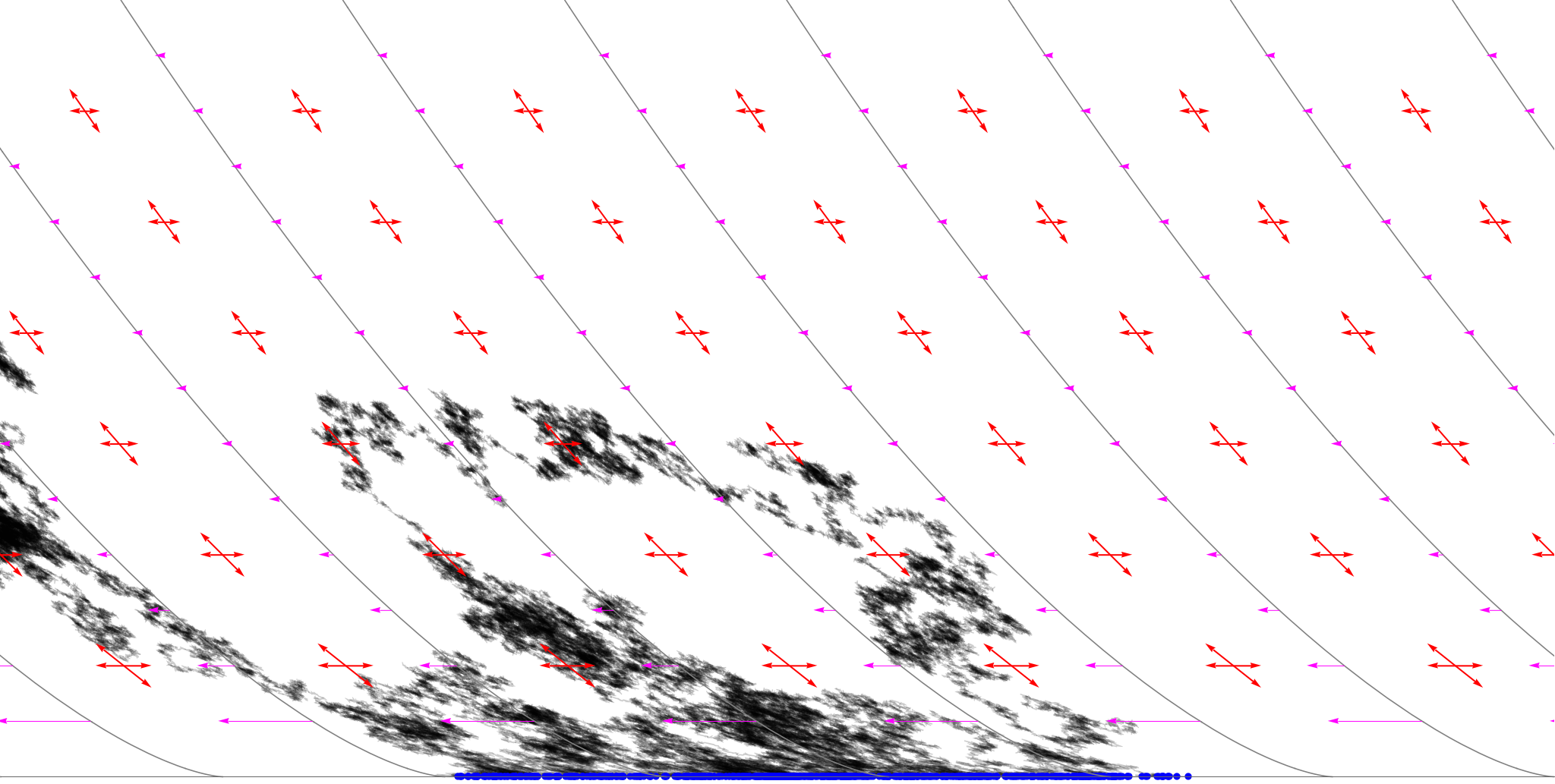}
\\
(a) & (b)
\end{tabular}
\caption{A diffusion, the trace of which is a two-sided non-symmetric stable process with index (a)~$0.5$, (b)~$1.5$.}
\label{fig:mix}
\end{figure}

\begin{figure}
\centering%\scriptsize
\begin{tabular}{cc}
\includegraphics[width=11cm]{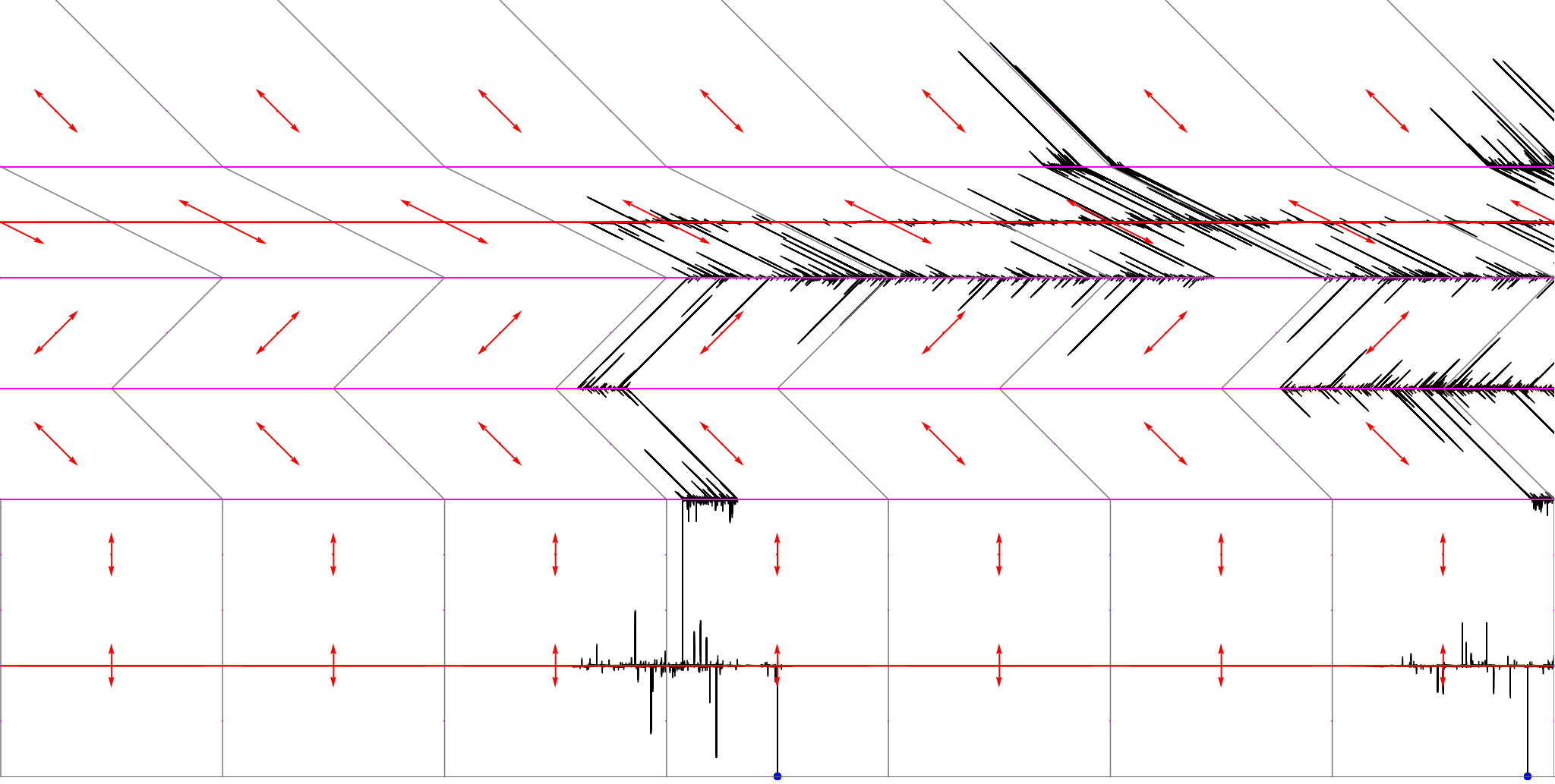}
&
\includegraphics[width=11cm]{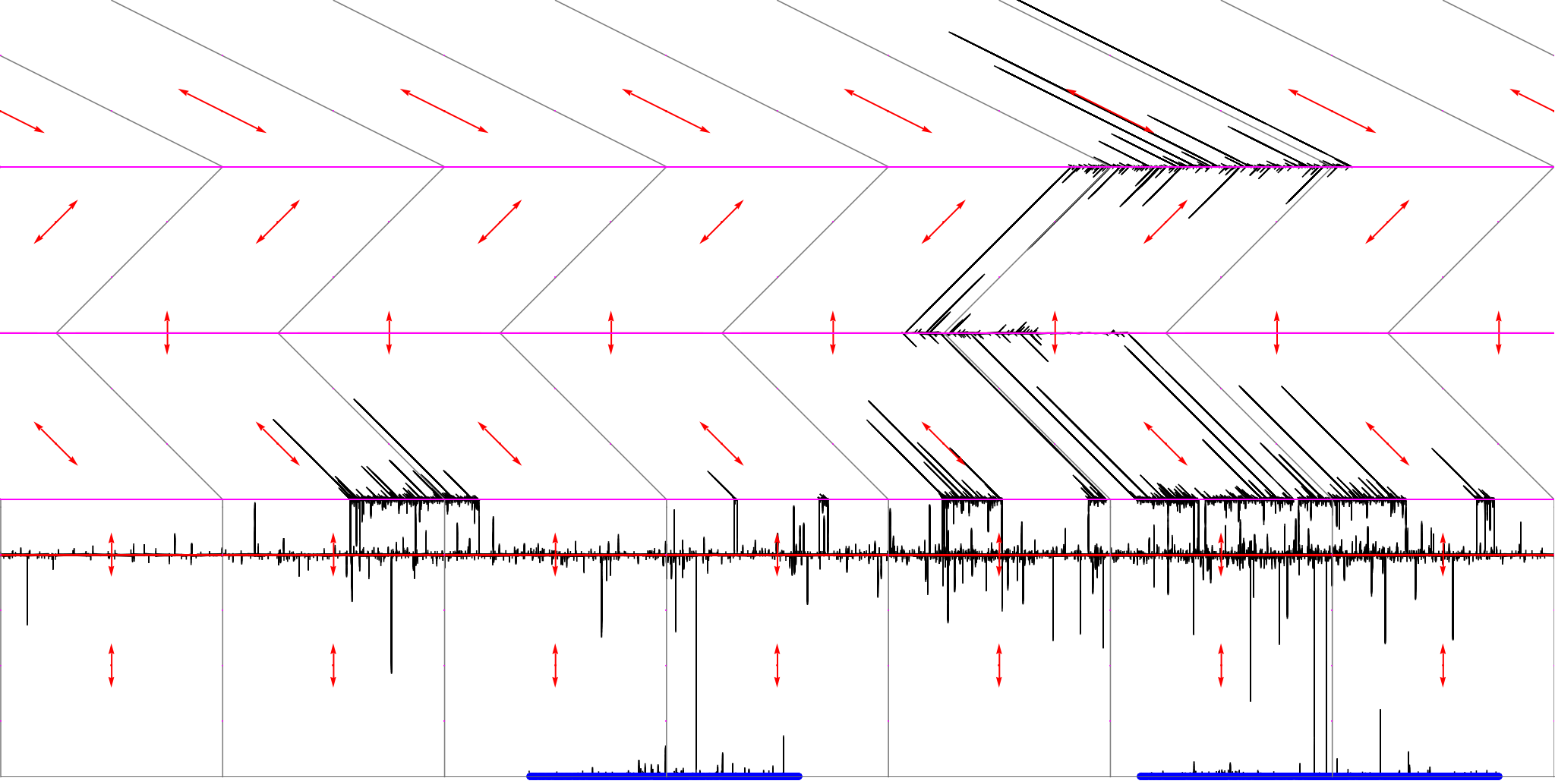}
\\
(a) & (b)
\end{tabular}
\caption{A degenerate diffusion, the trace of which is a meromorphic Lévy process (a)~without Brownian component, (b)~with Brownian component.}
\label{fig:phase}
\vspace*{1cm}
\end{figure}

\begin{figure}
\centering%\scriptsize
\begin{tabular}{cc}
\includegraphics[width=11cm]{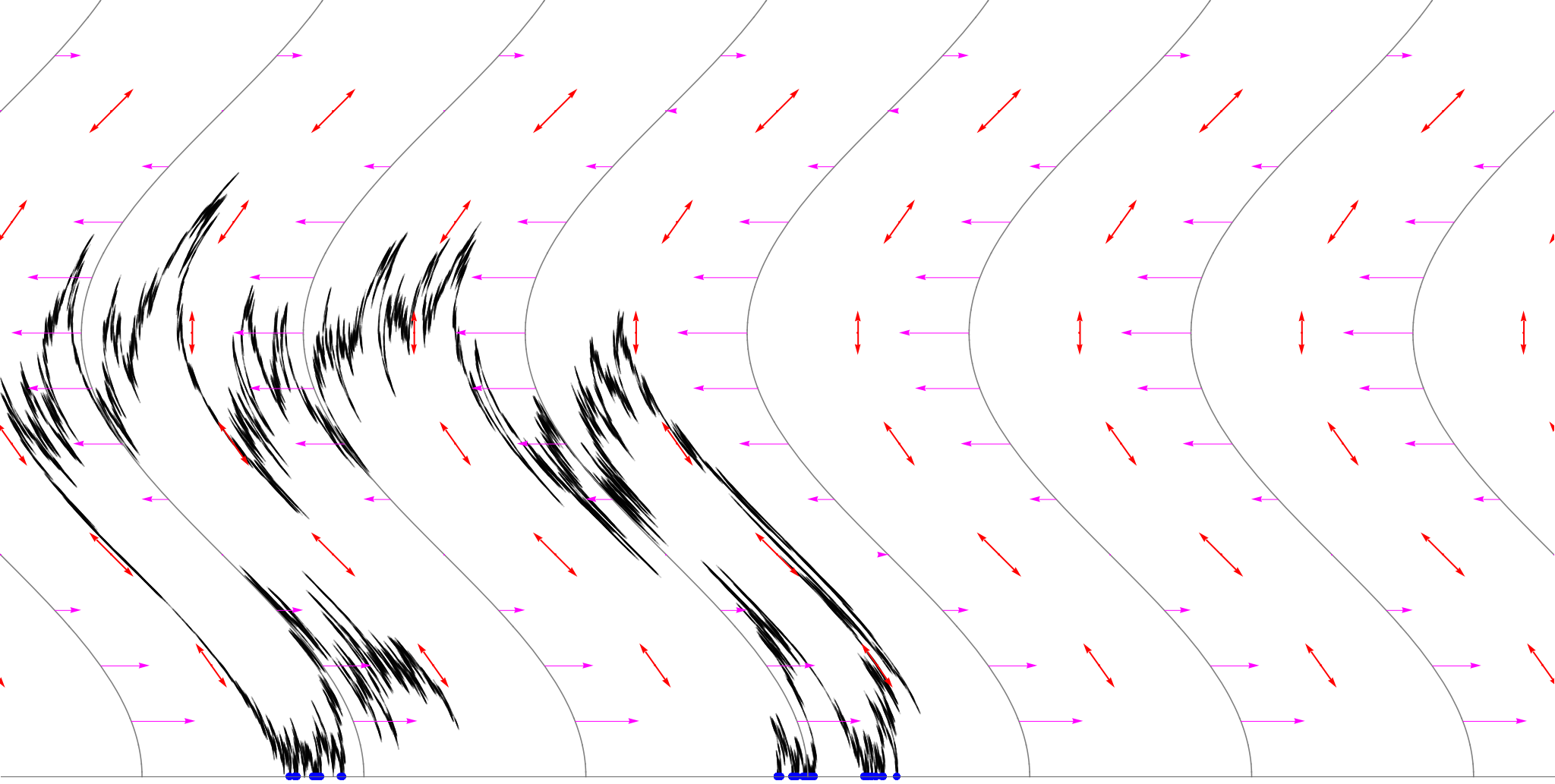}
&
\includegraphics[width=11cm]{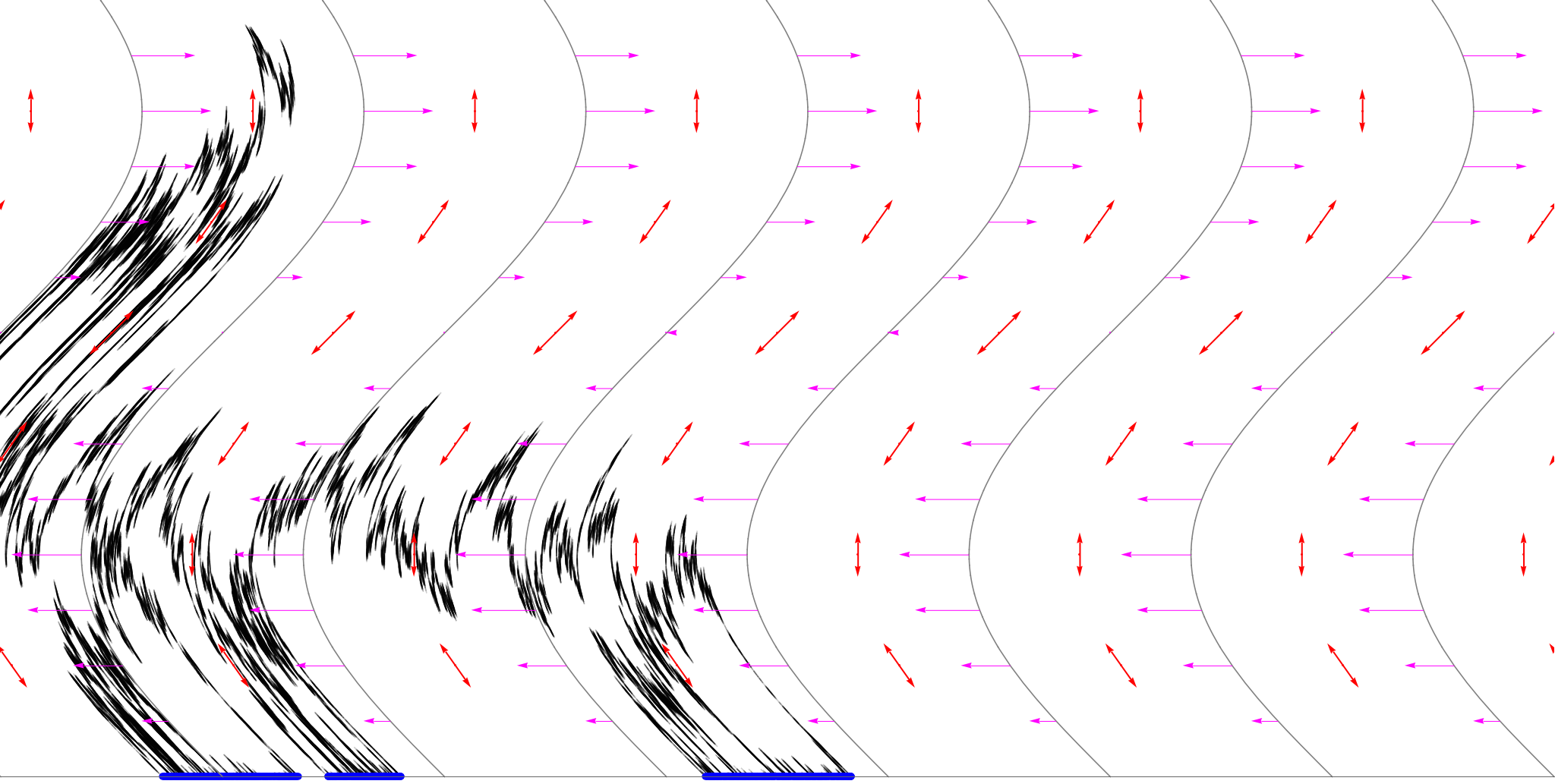}
\\
(a) & (b)
\end{tabular}
\caption{Two more degenerate diffusions, the trace of which has two-sided jumps.}
\label{fig:sin}
\end{figure}
\end{landscape}
}

\subsection{Trivial Lévy processes}

If $R = \infty$,
\formula{
 a(dy) & = 0 dy && \text{and} & b(y) & = 0 ,
}
then, of course, $X(t) = 0$ for all $t > 0$, and hence $Z(u) = 0$ for all $u > 0$. Similarly, if $R \in (0, \infty)$, and again $a(dy) = 0 dy$ and $b(y) = 0$, then $Z(u) = 0$ up to its finite life-time, which has an exponential distribution with mean $R$. In other words, $\gamma = 1 / R$ is the killing rate of $Z(u)$.

More generally, if $q \in \R$, $R = \infty$,
\formula{
 a(dy) & = 0 dy && \text{and} & b(y) & = -q ,
}
then $Z(u) = q u$ for all $u > 0$. Indeed: $X(t) = -q \dot{Y}(t) = -q Y(t) + q L_0(t)$, and hence
\formula{
 Z(u) & = X(L_0^{-1}(u)) = -q Y(L_0^{-1}(u)) + q u = q u .
}
Finally, with the above $a(dy) = 0 dy$ and $b(y) = -q$ in a finite interval $[0, R)$, $R \in (0, \infty)$, we have $Z(u) = q u$ up to its exponentially distributed life-time, with mean $R$.

These examples are completely degenerate: the diffusion $(X(t), Y(t))$ moves on a single profile line until it hits the boundary, and then it shifts to a parallel line according to the local time of $Y(t)$ at $0$, with velocity $q$; see Figure~\ref{fig:trivial}.

\subsection{Cauchy process and quasi-relativistic process}

If $R = \infty$,
\formula{
 a(dy) & = dy && \text{and} & b(y) & = 0 ,
}
then $(X(t), Y(t))$ is the usual two-dimensional Brownian motion in half-plane, and consequently $Z(u)$ is the Cauchy process; see Figure~\ref{fig:brown}(a). This has been first observed by Spitzer in~\cite{spitzer}, and it corresponds to the classical identification of the Dirichlet-to-Neumann map in a half-space with the square root of the Laplace operator.

If $m > 0$, $R = \infty$,
\formula{
 a(dy) & = \frac{1}{1 + 2 m y} \, dy && \text{and} & b(y) & = 0 ,
}
then $(X(t), Y(t))$ behaves in a similar way as the two-dimensional Brownian motion, but $X(t)$ runs at a lower pace: the larger the value of $Y(t)$, the slower the motion of $X(t)$; see Figure~\ref{fig:brown}(b). The corresponding trace process $Z(u)$ is sometimes called a (quasi-)relativistic Cauchy process, because its generator is the one-dimensional quasi-relativistic Hamiltonian of a free particle, $m - (m^2 - \partial_x^2)^{1/2}$; see~\cite{cms,kkm,ls} for more information.

\subsection{Stable Lévy processes}

Representation of symmetric stable Lévy processes as boundary traces of appropriate diffusions is due to Molchanov and Ostrovski, see~\cite{mo}. If $\alpha \in (0, 2)$, $R = \infty$,
\formula{
 a(dy) & = y^{2/\alpha - 2} dy && \text{and} & b(y) & = 0 ,
}
then $Z(u)$ is a symmetric $\alpha$-stable Lévy process. Note that for $\alpha < 1$, $X(t)$ runs faster near the boundary, while if $\alpha > 1$, $X(t)$ runs faster away from the boundary; see Figure~\ref{fig:sym}. For $\alpha = 1$, $X(t)$ is the Brownian motion, and $Z(u)$ is the Cauchy process, as discussed in the previous example; see Figure~\ref{fig:brown}(a).

Extension to non-symmetric stable Lévy processes seems to be new. If $\alpha \in (0, 2)$, $p \ge 0$, $q \in \R$, $R = \infty$,
\formula{
 a(dy) & = p^2 y^{2/\alpha - 2} dy && \text{and} & b(y) & = -q y^{1/\alpha - 1} ,
}
then $Z(u)$ is an $\alpha$-stable Lévy process. More precisely, if $\alpha \ne 1$, $Z(u)$ has Lévy measure
\formula{
 \nu(dx) & = (C_+ \ind_{(0, \infty)}(x) + C_- \ind_{(-\infty, 0)}(x)) \, \frac{1}{|x|^{1 + \alpha}} \, dx ,
}
where the constants $C_+$ and $C_-$ can be determined using the formulae given in Section~4.6 in~\cite{k:hx}:
\formula{
  C_+ & = \biggl\lvert \frac{A}{2 \cos \tfrac{\alpha \pi}{2}} - \frac{B}{2 \sin \tfrac{\alpha \pi}{2}} \biggr\rvert , &  \qquad C_- & = \biggl\lvert \frac{A}{2 \cos \tfrac{\alpha \pi}{2}} + \frac{B}{2 \sin \tfrac{\alpha \pi}{2}} \biggr\rvert ,
}
with
\formula{
 A + i B & = \begin{cases}
  \displaystyle \frac{-\Gamma(-\alpha) \Gamma(\alpha + \tfrac{(1 - \alpha) (p - i q)}{2 p})}{\Gamma(\alpha) \Gamma(\tfrac{(1 - \alpha) (p - i q)}{2 p})} \, (2 \alpha p)^\alpha & \text{if $p > 0$,} \\[1.5em]
  \displaystyle \frac{-\Gamma(-\alpha)}{\Gamma(\alpha)} \, e^{-(i \pi \alpha/2) \sign (q (1 - \alpha))} |q \alpha (1 - \alpha)|^\alpha & \text{if $p = 0$.}
 \end{cases}
}
When $\alpha = 1$, the Lévy measure of $Z(u)$ is simply
\formula{
 \nu(dx) & = \frac{p}{\pi} \, \frac{1}{|x|^2} \, dx ,
}
and $q$ is the drift coefficient of $Z(u)$.

If $p = 0$, the diffusion $(X(t), Y(t))$ is degenerate, and so it aligns closely to a profile line. It only moves between different profile lines due to the apparent drift, and, if $\alpha \ge 1$, also due to oblique reflection.

If $p = 0$, $q > 0$ and $\alpha < 1$, then $Z(u)$ is the $\alpha$-stable subordinator; see Figure~\ref{fig:sub}. On the other hand, if $p = 0$, $q > 0$ and $\alpha > 1$, then $Z(u)$ is the one-sided $\alpha$-stable process, with no negative jumps; see Figure~\ref{fig:onesided}. The case $p = 0$, $q > 0$ and $\alpha = 1$ corresponds to the trivial drift process, discussed in the first example in this section; see Figure~\ref{fig:trivial}(b).

When $p > 0$, the diffusion $(X(t), Y(t))$ is non-degenerate, and the corresponding $\alpha$-stable Lévy process $Z(u)$ has two-sided jumps; see Figure~\ref{fig:mix}.

\subsection{Meromorphic Lévy processes}

If $a(dy)$ is a purely atomic measure with finitely many atoms, and $b(y)$ is a piece-wise constant function with finitely many jumps, then the equations in Theorem~\ref{thm:ode} can be solved explicitly, and the characteristic exponent $\Psi(\xi)$ turns out to be (the restriction to $\R$ of) a meromorphic function of $\xi \in \C$. Conversely, if $\Psi(\xi)$ is the characteristic exponent of a Lévy process with completely monotone jumps and $\Psi(\xi)$ extends to a meromorphic function of $\xi \in \C$, then $a(dy)$ and $b(y)$ have the structure described above. Furthermore, the corresponding atoms of $a(dy)$ and jumps of $b(y)$ are easily found by an appropriate continued fraction expansion of $\Psi(\xi)$; we omit the somewhat technical details, and refer to~\cite{kw} for a closely related calculation in the symmetric case.

The class of Lévy processes obtained in this way is known under the name \emph{meromorphic Lévy processes}; see~\cite{kkp}. Two examples are given in Figure~\ref{fig:phase}.

Note that the diffusion $(X(t), Y(t))$ is degenerate, and the profile lines are piece-wise linear. The process $(X(t), Y(t))$ moves between different profile lines only at the jump points of $b(y)$ and the atoms of $a(dy)$.

\subsection{Two more examples}

We conclude the list of examples with two more degenerate diffusions, which correspond to Lévy processes with two-sided jumps. Suppose that $R = \infty$,
\formula{
 a(dy) & = 0 dy && \text{and} & b(y) & = -\sin y .
}
Then the diffusion $(X(t), Y(t))$ aligns to profile lines, where $x - \cos y$ is constant, and moves between them at a rate given by the apparent drift $(\cos y, 0)$. Note that the apparent drift takes both positive and negative values, so that the trace process $Z(u)$ has two-sided jumps. It can be easily verified that $Z(u)$ has finite variation, the Lévy measure of $Z(u)$ is finite on $(-\infty, 0)$ and infinite on $(0, \infty)$, and that $Z(u)$ has zero drift (the profile lines are orthogonal to the boundary); see Figure~\ref{fig:phase}(a). We omit the details.

In a similar way, consider now $R = \infty$,
\formula{
 a(dy) & = 0 dy && \text{and} & b(y) & = -\cos y .
}
Again, the diffusion $(X(t), Y(t))$ aligns to profile lines with $x + \sin y$ constant, and it moves between them at a rate given by the apparent drift $(-\sin y, 0)$. The trace process $Z(u)$ has two-sided jumps, finite variation, the Lévy measure of $Z(u)$ is infinite on $(-\infty, 0)$ and finite on $(0, \infty)$, and since $b(0^+) = -1$, the process $Z(u)$ has a unit positive drift; see Figure~\ref{fig:phase}(b). Again we omit the details.

%
%                            ---------- o ----------
%

\bigskip
\subsection*{Acknowledgments}
I thank Sigurd Assing, Jacek Małecki and Jacek Mucha for inspiring discussions on the subject of the present article. I also thank Tadeusz Kulczycki, from whom I learned about the concept of the boundary trace of a diffusion.

%
%                            ---------- o ----------
%

%
%                            ---------- o ----------
%

\end{document}